\documentclass[a4paper,12pt]{article}

\usepackage{enumitem}
\newcommand{\subscript}[2]{$#1 _ #2$}

\usepackage[utf8]{inputenc}
\usepackage{amstext}
\usepackage{amsfonts}

\usepackage{amsmath,amsthm}
\usepackage{textcomp}
\usepackage{amssymb}
\usepackage{amscd}
\usepackage{amsmath}
\usepackage{xcolor}
\usepackage{comment}

\makeatletter
\newtheorem*{rep@theorem}{\rep@title}
\newcommand{\newreptheorem}[2]{%
\newenvironment{rep#1}[1]{%
 \def\rep@title{#2 \ref{##1}}%
 \begin{rep@theorem}}%
 {\end{rep@theorem}}}
\makeatother

\newtheorem{Theorem}{Theorem}[subsection]
\newreptheorem{Theorem}{Theorem}
\newreptheorem{Corollary}{Corollary}
\newtheorem{Lemma}[Theorem]{Lemma}
\newtheorem{Proposition}[Theorem]{Proposition}
\newtheorem{Corollary}[Theorem]{Corollary}

\newtheorem{Definition}[Theorem]{Definition}
\newtheorem{Example}[Theorem]{Example}
\newtheorem{Remark}[Theorem]{Remark}

\newcommand{\seq}[1]{\left( #1 \right)}

\newcommand{\N}{\mathbb{N}}
\newcommand{\R}{\mathbb{R}}
\newcommand{\Z}{\mathbb{Z}}
\newcommand{\Q}{\mathbb{Q}}
\newcommand{\x}{\mathbf{x}}
\newcommand{\y}{\mathbf{y}}
\newcommand{\z}{\mathbf{z}}

\newcommand{\B}{\mathcal{L}}

\newcommand{\m}[2]{\mathcal {#1}_ {#2}}
\newcommand{\invlim}[2]{\lim\limits_{\longleftarrow}\{#1,#2\}}

\renewcommand\labelenumi{(\roman{enumi})}
\renewcommand\theenumi\labelenumi

\begin{document}

\title{Shadowing, Finite Order Shifts and Ultrametric Spaces}
\author{Udayan B. Darji\\
\and  
Daniel Gon\c{c}alves \thanks{partially supported by Conselho Nacional de Desenvolvimento Cient\'ifico e Tecnol\'ogico (CNPq) and Capes-PrInt grant number 88881.310538/2018-01 - Brazil.} \\ 
\and 
Marcelo Sobottka \thanks{partially supported by Conselho Nacional de Desenvolvimento Cient\'ifico e Tecnol\'ogico (CNPq) grant number 301445/2018-4 - Brazil,  and Capes-PrInt grant number 88881.310538/2018-01 - Brazil.} \\}

\date{ }
\maketitle

\begin{abstract}

Inspired by a recent novel work of Good and Meddaugh, we establish fundamental connections between shadowing, finite order shifts, and ultrametric complete spaces. We develop a theory of shifts of finite type for infinite alphabets. We call them shifts of finite order.
We develop the basic theory of the shadowing property in general metric spaces, exhibiting similarities and differences with the theory in compact spaces. We connect these two theories in the setting of zero-dimensional complete spaces, showing that a uniformly continuous map of an ultrametric complete space has the finite shadowing property if, and only if, it is an inverse limit of a system of shifts of finite order satisfying the Mittag-Leffler Condition. Furthermore, in this context, we show that the shadowing property is equivalent to the finite shadowing property and the fulfillment of the Mittag-Leffler~Condition in the inverse limit description of the system. As corollaries, we obtain that a variety of maps in ultrametric  spaces have the shadowing property, such as similarities and, more generally,  maps which themselves, or their inverses, have Lipschitz constant 1.  Finally, we apply our results to the dynamics of $p$-adic integers and $p$-adic rationals.
\footnote{2020 {\em Mathematics Subject Classification:} Primary 37B65, 37B10 Secondary
11S82, 37P05\\ 
{\em Keywords:} Shadowing, finite order shift spaces, ultrametric spaces, $p$-adics, inverse limits, Polish space, non-Archimedean dynamical systems}
\end{abstract}

\newpage
\tableofcontents
\newpage

\section{Introduction}
Very recently Good and Meddaugh \cite{GMinventiones} made a fundamental structural connection between the notions of shifts of finite type and the shadowing property. A classical result of Walters \cite{WaltP} states that a shift has the shadowing property if, and only if, it is a shift of finite type. The fundamental result in \cite{GMinventiones} shows that a continuous function $f:X\rightarrow X$, where $X$ is a compact totally disconnected space, has the shadowing property if, and only if, the system $(X,f)$ is conjugate to the inverse limit of a directed system consisting of shifts of finite type and satisfying the Mittag-Leffler Condition. As tempting as it is to reiterate the history and the importance of shifts of finite type and shadowing, we refrain from it here and let the reader browse the beautiful exposition in \cite{GMinventiones}. We do, however, highlight that in the compact setting finite shadowing and infinite shadowing agree. This fact has boosted the study of the shadowing property, as it allows for an interface between theoretical and numerical results.

Our goal in this paper is to start a systematic study of the shadowing property in zero-dimensional (not necessarily compact) metric spaces. As we pointed out above, a key feature to be understood is the equivalence between finite and infinite shadowing. In our work, we show that this equivalence does not hold for general continuous functions on locally compact spaces (see Example~\ref{FinShadowEx}), but it does on uniformly compact spaces (Proposition~\ref{UnifLCThm}). In the setting of general metric dynamical systems the situation is more complicated. We have, in fact, completely characterized when finite shadowing and infinite shadowing agree on dynamical systems (Theorem~\ref{joelho}), but this requires an inverse limit description of the system, which proves to be a fundamental result in the study of the shadowing property (as are the results of \cite{GMinventiones} in the compact case). 

As we hinted above, we will connect the shadowing property of a dynamical system   $(X,d,f)$, $X$ zero-dimensional and complete, with simple objects such as shifts. Our general approach is inspired by \cite{GMinventiones}, however, it departs from it in several important ways.




First, we must find an appropriate concept of ``shift of finite type" in the noncompact setting. There are many approaches to shifts over infinite alphabets, e.g., \cite{MR3938320, MR3207861}. We will focus on the approach which takes the full shift space over an alphabet to be the product space, with product topology, and the usual shift map. These shifts have been studied extensively, e.g., \cite{MR1484730}. In Section~2, we introduce the notion of shift of finite order, which is, in the setting of countable alphabets, an analog of a shift of finite type.

Another important difference, as we have already mentioned, is the notion of shadowing in the noncompact framework. In fact, in \cite{GMinventiones} it is shown that in the setting of compact metric spaces, shadowing is a topological property. Using this fact the authors define a meaningful notion of shadowing in an arbitrary compact space. In the noncompact setting things are very different, as finite shadowing and shadowing are not equivalent properties, the former being more general. Moreover, two dynamical systems, one with the shadowing property and the other without, may be conjugate to each other, see Example~\ref{FinShadowEx}. We shall see, in Proposition~\ref{UnifCong}, that uniform conjugacy remedy this problem, i.e., two uniformly conjugate  dynamical systems either both have the shadowing property or neither has it. Of course, in the compact metric space setting uniform conjugacy and conjugacy are equivalent.


In our work, the metric on the space plays a significant role. We therefore introduce the notions of ``defining sequence" and ``tame defining sequence", see Definitions~\ref{nintendoDS} and \ref{tameDS}. A defining sequence $\{\mathcal {U}_n \}$ can be thought of as a sequence of partitions of a zero-dimensional space $X$, where the partitions get finer as $n$ increases. A tame defining sequence is a defining sequence where the mesh of the diameter of the partition goes to zero as $n$ increases and, at the same time, given a partition in the sequence any two elements of it are a fixed distance apart. These definitions may seem contrive, however, we show that a space admits a tame defining sequence if and only if its metric is uniformly equivalent to an admissible ultrametric. Our main theorem which connects shadowing and shifts of finite order is the following.

\begin{repTheorem}{ThmInvLimSh}
Let $X$ be a metric space with a complete tame defining sequence and $(X,d,f)$ be a dynamical system with the finite shadowing property.  Then, $f$ is conjugate to the inverse limit of a sequence of 1-step shifts on a countable alphabet. More precisely, $(X,d,f)$ is conjugate to $   (\invlim{\hookrightarrow}{\m{PO}{}(\m{U}{n})},\sigma^*)$. Moreover, if $f$ is uniformly continuous, then the conjugacy can be made uniform. 
\end{repTheorem}

Extending the classical result of Walters mentioned earlier, we show that for shifts on a countable alphabet, the finite shadowing property, the shadowing property, and being a shift of finite order are equivalent, see Proposition~\ref{shadow-finiteshadow-porder}. (This was also proved independently by Meddaugh and Raines in \cite{meddaughRaines}.) In the context of uniformly continuous dynamical systems with  complete tame defining sequences, we show that the equivalence between the shadowing property and the finite shadowing property depends on the fulfillment of the Mittag-Leffler Condition in the inverse limit description of a natural system, see Theorem~\ref{joelho}.

We also show that the inverse limit of metric dynamical systems with the shadowing property also has the shadowing property, provided that the inverse system satisfies the Mittag-Leffler Condition, see Theorem~\ref{caipirinha}. Putting this together with the above theorem, we have the following corollary. 

\begin{repCorollary}{CorMain}
Let $(X,d)$ be a metric space with a  complete tame defining sequence and $f : X \rightarrow X $ be a uniformly continuous map. Then, $f$ has the shadowing property if, and only if, $f$ is uniformly conjugate to an inverse limit dynamical system, in which the associated inverse system consists of a sequence of 1-step shifts on a countable alphabet, with uniformly continuous bonding maps, and satisfies the Mittag-Leffler Condition.
\end{repCorollary}

We apply above results and the techniques developed in the article to ultrametric  spaces. An ultrametric space is a space where the triangle inequality is strengthened to $d(x,z) = \max \{d(x,y),d(y,z)\}$. Ultrametric spaces have topological dimension zero.
They naturally appear in a variety of places, including general topology, mathematical logic, theoretical computer science, $p$-adic dynamics, and theoretical biology. For example, results regarding Lipschitz and uniformly continuous and Borel reducibilities are studied in \cite{ CamerloMarconeMottoRos, DBLP:books/daglib/p/RosS14}, generic elements in isometry groups are described in \cite{Malick}, Polish ultrametric spaces on which each Baire one function is first return recoverable are characterized in \cite{Duncan} and locally contracting maps on perfect Polish ultrametric spaces are studied in \cite{vesnik}.

As ultrametric spaces are zero-dimensional, they have a basis of clopen sets. Standard notions such as being Lipschitz or being an isometry are often localized. For example, if we consider the balls of radius $\varepsilon >0$ for a fixed $\varepsilon$,  we obtain a partition of an ultrametric space. Often maps in question are Lipschitz or an isometry at this local level. We call this local behavior eventual, see Definition~\ref{feijao}. The following corollaries establish the shadowing property of maps of these types. 

\begin{repCorollary}{camarao}
Let $(X,d)$ be a complete ultrametric space. If $f:X \rightarrow X$ is an eventually  $1$-Lipschitz map, then $f$ has the shadowing property. 
\end{repCorollary}

\begin{repCorollary}{almoco}
Suppose that $(X,d)$ is an  ultrametric space with the additional property that for some $\varepsilon >0$, all balls of radius $\varepsilon$ are compact.  Let $f:X\rightarrow X$ be an invertible uniformly continuous map.
\begin{enumerate}
    \item If $f^{-1}$ is an eventually $1$-Lipschitz map, then $f$ has the shadowing property.
   \item If $f^{-1}$ is uniformly continuous and $f$ is an eventual similarity, then $f$ has the shadowing property. 
\end{enumerate} 
\end{repCorollary}

In particular, all similarities in compact ultrametric spaces have the shadowing property. Of course, for contractions and dilations, this is well known in general Polish spaces. What may be somewhat surprising is that the identity map has the shadowing property. Of course, such is never the case in a connected metric space with more than one element. In Example~\ref{exqueijo}, we point out that shifts which are not of finite order are Lipschitz, but do not have the shadowing property. Hence, our results are, in some sense, sharp.

An important class of ultrametric dynamical systems consists of the $p$-adic dynamics, which has both practical and theoretical applications (we refer the reader to the Aims and Scope section of the journal ``$p$-Adic Numbers, Ultrametric Analysis and Applications" for a comprehensive description of fields which intersect with $p$-adic dynamics). Among developments in the area we mention that minimal polynomial dynamics on the set of 3-adic integers is studied in \cite{10.1112/blms/bdp003}, strict ergodicity of affine $p$-adic dynamical systems on $\Z_p$ is described in \cite{FAN2007666}, minimal decomposition of $p$-adic homographic dynamical systems is studied in \cite{fanfan}, and shadowing and stability in $p$-adic dynamics are studied in the recent paper \cite{BCA}. We are particularly interested in the results of \cite{BCA}. As consequences of our results on ultrametric  spaces, we recover many of the shadowing results in \cite{BCA}, in addition to answering a question left open. More precisely, we prove the following.

\begin{repCorollary}{Cortakethat}
The three statements below hold.

\begin{enumerate}
    \item \cite[Theorem~1]{BCA} 
If $f: \mathbb{Z}_p \to \mathbb{Z}_p$ is a $(p^{-k},p^{m})$ locally scaling function, where $1 \leq m \leq k $ are integers, then $f$ has the shadowing property.
 \item  \cite[Proposition~18]{BCA} 
If $f: \Z_p \to \Z_p$ is a $1$-Lipschitz map, then $f$ has the shadowing property.
\item  \cite[Remark~19]{BCA} 
If  $f: \mathbb{Q}_p \to \mathbb{Q}_p$ is a $1$-Lipschitz map, then $f$  has the shadowing property.
\end{enumerate}
\end{repCorollary}

In the final remarks of this introduction, we would like to point out that dynamics outside the realm of compact metric space is a thriving area. For instance, strong orbit equivalence and topological equivalence of locally compact Cantor minimal systems are studied in \cite{Danilenko, Matui} and topological full groups as invariants for groupoids associated to locally compact spaces (and with connections to graph and ultragraph C*-algebras) are studied in \cite{ CGW, NylandOrtega}. The theory of chaos (including Li-Yorke Chaos, Devaney chaos, distributional chaos) is also explored in the context of locally compact dynamics, see for example \cite{DAI20175521, brunodaniel, GONCALVES2020102807, raines}. 
Our result provides a strong connection, via inverse limits, between the theory of shadowing on dynamical systems and the theory of shift spaces with the product topology. The inverse limit description of a dynamical system has various advantages, as it is often applied in the computation of invariants for the system (for example, \v{C}ech cohomology of a tiling dynamical system is computed from its inverse limit description (see \cite{AndersonPutnam, RamirezGonc}). We hope our results and techniques apply to shadowing in various non compact settings mentioned above.

The paper is organized as follows. After this paragraph we set up notation that will be used throughout the paper. In Section~2, we define and develop symbolic dynamics on countable alphabets pertinent to our results. We also develop the basic theory of the finite shadowing property and the  shadowing property in the setting of locally compact metric spaces. Notions of defining sequence, ultrametric spaces, and $p$-adics, the settings in which our main results reside, are developed in Section~3. Section~4 concerns shadowing, inverse limits, and proofs of our main results. In the final Section~5, we give applications of our results.

Throughout our work $\N$ and $\N_+$ denote the set of non-negative integers, and positive integers, respectively.  For the sake of ease of readability, we denote points of a space $X$ in bold style (e.g., $\x,\y,\z\in X$). 

 Each of our spaces $X$ is assumed to be nonempty, Hausdorff and
second countable, or, equivalently, a separable, metrizable topological
space. A {\em metric space} is a pair $(X, d)$ consisting of a space together
with a choice of an {\em admissible metric}, i.e., one whose induced topology is
that of the space. A space is {\em Polish} when it admits a complete metric.
It is well-known that a nonempty subset of a Polish space is Polish if
and only if it is a $G_\delta$ subset. A space is zero-dimensional when the
clopen subsets form a basis for the space. A {\em dynamical system} $(X,d,f)$ is a metric space $(X,d)$ equipped with a continuous map $f: X\rightarrow X$. If $\{(X_i,d_i)\}_{i\in \N}$ is a sequence of uniformly bounded metric spaces we define the {\em product metric} as 
\begin{equation}\label{productmetric}
    d_\Pi(x,y)=\max_i  \left \{\frac{d_i(x_i,y_i)}{i+1} \right \}.
\end{equation}
We remark for future use that if $h_i:(X_i,d_i) \rightarrow (X_i', d_i')$ are (uniformly) continuous maps, then the product map $\prod_i h_i $ is a (uniformly) continuous map from $(\prod X_i, d_\Pi)$ to $(\prod X_i', d'_\Pi)$. Moreover, if $h_i$ is 1-Lipschitz map, then so is the product map $\prod_i h_i $.

\section{Shadowing and Symbolic Dynamics}

\subsection{Shadowing in metric spaces}
In this section, we recall the definition of shadowing and develop some of its properties in metric spaces.

\begin{Definition}\label{lasagna} Let $X$ be a metric space, $f:X\to X$ be a map, and $I$ be an initial segment of $\N$. We say that,
\begin{enumerate}
\item The sequence $(f^n(\x))_{n\in\N}$ is the \emph{orbit}, or \emph{trajectory}, of the point $\x\in X$.

\item A (finite or infinite) sequence $(\x_n)_{n\in I}\in X$ is a \emph{(finite or infinite) $\delta$-chain}, or \emph{$\delta$-pseudo-orbit (trajectory)}, if $d(f(\x_n), \x_{n+1})<\delta$ for all $n < \sup (I)$.

\item A point $\x\in X$ \emph{$\varepsilon$-shadows} a (finite or infinite) sequence $(\x_n)_{n\in I}$ if $d(f^n(\x),\x_n)<\varepsilon $, for all $n\in I$.
\end{enumerate}
\end{Definition}

\begin{Definition}\label{mignon}
We say that a dynamical system $(X,d, f)$ has the \emph{finite shadowing property} if, for any $\varepsilon>0$, there exists $\delta>0$ such that any finite $\delta$-pseudo-orbit is $\varepsilon$-shadowed by some point. 
We say that $(X,d, f)$ has the \emph{shadowing property} if, for any $\varepsilon>0$, there exists $\delta>0$ such that any infinite $\delta$-pseudo-orbit is $\varepsilon$-shadowed by some point.
\end{Definition}

There are other concepts of shadowing in topological dynamics,  such as $\ell_p$ shadowing, asymptotic shadowing,  $s$-limit shadowing, etc. For a study of the relations between these notions, the reader is referred to \cite{goodoprocha}.
We also point out that a different notion of shadowing was studied in \cite{DasLeeEtal} and a spectral decomposition theorem was proved. For general references on shadowing, we suggest the classical texts  \cite{palmer, Pilyugin}.

It was shown in \cite[Lemma 6]{GM} that for compact metric spaces shadowing is a topological concept, i.e., shadowing is independent of the metric as long as the topology is the same. The following simple example shows such is not the case for locally compact spaces.

\begin{Example}\label{ExMetricDep}
There exists a locally compact space $X$, a continuous function $f:X \rightarrow X$, and two complete metrics on $X$ which generate the same topology, one yielding that $f$ has the shadowing property and the other not.
\end{Example}
\noindent
{\em Construction:} Our space $X$ will be a countable subset of $\R$. Let $X=\{\x_n:n\in \Z\}$, where $\x_n = n$ for $n\geq 0$ and, for $n>0$, $\x_{-n}=n+\frac{1}{n+1}$.  Let $f:X \rightarrow X$ be defined by $f(\x_n) = \x_{n+1}$, for every $n\in \Z$. 
It is clear that $f$ is a homeomorphism of $X$. It is also clear that if one uses the discrete metric on $X$, in which the distance between any two distinct points is 1, then $f$ has the shadowing property with respect to this metric. It is also clear that this metric generates the same topology as the topology induced by the metric on $\R$. 

We will show that with respect to the usual metric of $\R$, $f$ does not have the shadowing property. To this end, notice that the orbit of every point in $X$ is unbounded. Let $\varepsilon = 1/2$ and $\delta >0$. We will construct a $\delta$-pseudo-trajectory that is bounded, and hence cannot be $\varepsilon$ shadowed by an orbit. Let $N$ be a positive integer such that $\frac{1}{N}<\delta$. We define a $\delta$-pseudo-trajectory $\{\z_n\}_{n=0}^{\infty}$ as follows: let $\z_i = \x_i$ for $ 0 \le i < N$, $\z_N=\x_{-N}$ and $\z_{N+k} = \x_{-N+k}$ for $1 \le k \le N$. Notice that we have defined $\z_i$ for $i=0,\ldots 2N$, and $\z_{2N}=\z_0$. So we extend this definition periodically, that is, for $i>2N$, let $\z_i = z_{i \mod 2N}$. Then $\{\z_n\}_{n=1}^{\infty}$ is a bounded, periodic $\delta$-pseudo-orbit that can not be $\varepsilon$-shadowed by a trajectory.
\qed

\begin{Remark}
Notice that the (finite) shadowing property is clearly a uniform notion, that is, if $(X,d_1,f)$ has the (finite) shadowing property, and $d_1$ is uniformly equivalent to $d_2$, then $(X,d_2,f)$ also has the (finite) shadowing property. Recall that $d_1$ and $d_2$ are uniformly equivalent if the identity maps $i:(X,d_1)\rightarrow (X,d_2)$ and $i:(X,d_2)\rightarrow (X,d_1)$ are uniformly continuous.
\end{Remark}

As the above example shows, conjugacy is not enough to preserve shadowing among  noncompact metric dynamical systems. We need a stronger version of conjugacy, namely uniform conjugacy, which we define below.

\begin{Definition} 
    Suppose $(X,d_1,f)$ and $(Y,d_2,g)$ are dynamical systems. We say that \emph{ $(X,d_1,f)$ and $(Y,d_2,g)$ are uniformly conjugate} if there is a surjective homeomorphism $h:X \rightarrow Y$, with $h$ and $h^{-1}$ uniformly continuous, such that $h\circ f = g \circ h$.
\end{Definition}

 The following proposition is straightforward, but will be used throughout our paper. We leave the proof to the reader.

\begin{Proposition}\label{UnifCong}
Suppose that $(X,d_1,f)$ and $(Y,d_2,g)$ are uniformly conjugate dynamical systems. If $(X,d_1,f)$ has the (finite) shadowing property, then so does $(Y,d_2,g)$.
\end{Proposition}

Since some of our key results depend on completeness of the space studied, we describe below the relation between the (finite) shadowing property in a uniformly continuous dynamical system and its completion.
The next two results were kindly pointed out to us by the anonymous referee.

\begin{Theorem} Suppose that $(X, d, f)$ is a dynamical system with $f$ uniformly
continuous. Let $(\hat X,\hat d)$ be the completion of $(X,d)$ and $\hat f$ be the extension of
$f$ to $(\hat X,\hat d)$. If $(X, d, f)$ has the (finite) shadowing property, then $(\hat X,\hat d,\hat f)$ does as well.
\end{Theorem}

\begin{proof} 
First, notice that $\hat f$, the extension of $f$ to $(\hat X, \hat d)$, is also uniformly continuous. We show below that $(\hat X, \hat d, \hat f)$ has the shadowing property (the proof of the finite shadowing property is analogous).

Given $\varepsilon > 0$, let $0<\delta<\frac{\varepsilon}{3}$ be such that any $\delta$-pseudo-orbit of $f$ is $\frac{\varepsilon}{2}$-shadowed. Let $0<\delta_1<\frac{\delta}{3}$ be such that if $\hat d(x,y)<\delta_1$ then $\hat d(\hat f(x),\hat f(y))<\frac{\delta}{3}$. Assume that $\{\hat x_i\}$ is a $\frac{\delta}{3}$-pseudo-orbit of $\hat f$. Choose, for every $i\in\N$,
$x_i \in X$ with $\hat d(x_i, \hat x_i) < \delta_1$. Then, since
\begin{align*}
   d(f(x_i), x_{i+1})&=\hat d(\hat f(x_i), x_{i+1})\\
 &\leq \hat d(\hat f(x_i), \hat f(\hat x_i)) + \hat d( \hat f(\hat x_i), \hat x_{i+1}) + \hat d(\hat x_{i+1}, x_{i+1}) < \delta,
\end{align*}
there exists $z\in X$ which $\frac{\varepsilon}{2}$-shadows $\{x_i\}$ and so it $\varepsilon$-shadows
$\{\hat x_i\}$.

\end{proof}

The converse of the previous theorem is not true in general, as we show in the example below. 

\begin{Example}
There exists a dynamical system $(X,d,f)$, such that its completion $(\hat X, \hat d, \hat f)$  has the  shadowing property, but $(X,d,f)$ does not. \end{Example}
\noindent
{\em Construction:} 
 Take, for instance, $(\hat X, \hat d)$ as a complete metric space, and $\hat f:\hat X\to\hat X$ as a uniformly continuous map. Suppose additionally that  $(\hat X, \hat d, \hat f)$ has the shadowing property, is topologically transitive and that the set $X$ of all of its periodic points is dense in $\hat X$ (for example, one could take an Anosov diffeomorphism on a manifold or simply the doubling map on the circle \cite{coven}). It follows that $( X,  d, f)$, where $d:=\hat d_{|_X}$ and $f:=\hat f_{|_X}$, is a uniformly continuous dynamical system whose completion is $(\hat X, \hat d, \hat f)$, but $( X,  d, f)$ does not have the shadowing property, as we show below.
 
 Given $C_1,C_2\subset X$, two distinct periodic orbits, let $a:=d(C_1,C_2):=\min\{d(x,y):\ x\in C_1,\ y\in C_2\}>0$ and set $\varepsilon=a/3$. For any given $\delta>0$, let $C_3\subset X$ be a periodic orbit such that $d(C_1,C_3)<\delta$ and $d(C_3,C_2)<\delta$ (the existence of such $C_3$ is ensured by the fact that $(\hat X, \hat d, \hat f)$ is topologically transitive and $X$ is dense in $\hat X$). Let $r\in C_1$, $s,t\in C_3$ and $u\in C_2$ be such that $d(C_1,C_3)=d(r,s)$ and $d(C_3,C_2)=d(t,u)$. Take $\ell$ and $m$ positive integers such that $f^\ell(r)=r$ and $f^m(s)=t$, and define the sequence $\{x_i\}$ given by: 
$$x_i:=\left\{\begin{array}{lcl}f^i(r)               &,\ if& 0\leq i \leq \ell-1\\
                                                f^{i-\ell}(s)  &,\ if& \ell\leq i\leq \ell+ m-1\\
                                                f^{i-\ell-m}(u)  &,\ if& i\geq \ell+ m
                      \end{array}\right.$$

It follows that $\{x_i\}$ is a $\delta$-pseudo-orbit, but, since all the orbits in $X$ are periodic,  $\{x_i\}$ cannot be $\varepsilon$-shadowed by any point of $X$. 

\subsection{Symbolic Dynamics of Countable Alphabets}

We start the section by recalling the definition of shift spaces.

\begin{Definition}\label{defnshifts} Let $A$ be a non-empty countable set endowed with the discrete topology. The \emph{(one-sided) full shift on the alphabet $A$} is the set \[A^\N:=\{\x=(x_i)_{i\in\N}:\ x_i\in A\ \forall i\in\N\},\]
with the associated prodiscrete topology, i.e., the product of discrete topology. Moreover, on $A^\N$, we will use the product metric given on Equation~(\ref{productmetric}), which is equivalent to the usual metric for shift spaces, and in this context can be written as $d(\x,\y) = (i+1)^{-1}$, where $i$ is the least integer where $\x_i \neq \y_i$.
\end{Definition}

\begin{Definition} The \emph{shift map} is the map $\sigma:A^\N\to A^\N$ defined, for all $\x=(x_i)_{i\in\N}\in A^\N$, by $$\sigma(\x):=(x_{i+1})_{i\in\N}.$$
\end{Definition}

\begin{Definition} We say that $X\subseteq A^\N$ is a \emph{(one-sided) shift (or subshift) space} if it is a closed set and it is shift invariant (that is, $\sigma(X)\subseteq X$).
\end{Definition}

In a general space $X$, we will denote a sequence indexed by $I \subseteq \N$ as $(\x_n)_{n\in I}$. In the particular case that each $\x_n$ is itself a sequence, we will use the notation $\x_{n,i}$ to denote the $i^{th}$ entry of the sequence $\x_n$.

Given a shift space $X\subseteq A^\N$ and $n\in\N_+$, we define $\B_n(X)$ as the set of all words of length $n$ that appear in some sequence of $X$, that is, $$\B_n(X):=\{(a_0\ldots a_{n-1})\in A^n:\ \exists \ \x\in X \text{ s.t. } (x_0\ldots x_{n-1})=(a_0\ldots a_{n-1})\}.$$ Clearly $\B_n(A^\N)=A^n$. We also define $\B_0(X):=\{\varepsilon\}$, where $\varepsilon$ stands for the empty word.
The language of $X$ is the set $\B(X)$, which consists of all finite words that appear in some sequence of $X$, that is,
$$\B(X):=\bigcup_{n\in\N}\B_n(X).$$

A well-known equivalent way to define shift spaces is given in terms of forbidden words. Given $F\subset \B(A^\N)$, we can define $X_F\subseteq A^\N$ as the set of all sequences in $A^\N$ which do not contain any word of $F$. In particular, one can check that given a shift space $X\subseteq A^\N$,  setting $F:=\B(A^\N)\setminus \B(X)$, we have that $X=X_F$.

\begin{Definition}\label{defn:p-shifts} A shift space $X$ is said to {\em have order $p$}, for some $p\in\N$, if there is a set of words $F$ such that $X= X_F$ and every word in $F$ has length $p$. Following the standard usage in symbolic dynamics, we say that a shift is a 1-step shift, if it is a shift of order 2. 
\end{Definition}

The next proposition corresponds to \cite[Theorem 2.1.8]{lindMarcus} and gives an alternative characterization of shift spaces of finite order. Although in \cite{lindMarcus} it is stated for shift spaces over finite alphabets, the proof given there also works for shift spaces over infinite alphabets. Such a characterization will be used in Proposition~\ref{shadow-finiteshadow-porder} to characterize shift spaces with the shadowing property.  

\begin{Proposition}\label{prop:p-order_characterization}
$X$ is shift of order $p$ if, and only if, for all $\mathbf{u},\mathbf{v},\mathbf{w}\in \B(X)$, with $|\mathbf{v}| = p-1$ and $\mathbf{uv},\mathbf{vw}\in \B(X)$, it follows that $\mathbf{uvw}\in \B(X)$.
\end{Proposition}

\begin{Remark} Notice that if $A$ is finite, then the class of shift spaces with finite order always coincides with the class of shift spaces of {\em finite type} (SFT), i.e., those subshifts that can be obtained from finite sets of forbidden words. On the other hand, if $A$ is infinite, then the class of shift spaces of finite type is strictly contained in the class of shift spaces of finite order.
\end{Remark}

\subsection{Finite Shadowing vs Shadowing}
As we are dealing with possibly locally compact sets it is not straightforward that finite and infinite shadowing agree. We, therefore, provide a sufficient condition on general dynamical systems for finite and infinite shadowing to agree and also show that a shift space has finite order if, and only if, it has the (finite) shadowing property. In Section~\ref{ovelha}, in the context of Polish dynamical systems, we will describe precisely the relation between finite shadowing and shadowing, see Theorem~\ref{joelho} and Proposition~\ref{balance}.

It is well known that, for compact spaces, shadowing and finite shadowing are equivalent, see \cite[Lemma~1.1.1]{Pilyugin} for example. Following the general idea from there, below we show this fact for uniformly locally compact spaces.

\begin{Definition} We say that a metric space is {\em uniformly locally compact}, if there exists $\varepsilon>0$ such that for any $\x\in X$ the open ball centered at $\x$ with radius $\varepsilon$ is contained in a compact set.
\end{Definition}

\begin{Proposition}\label{UnifLCThm} Let $(X,d, f)$ be a dynamical system where $X$ is uniformly locally compact. Then $(X,d, f)$ has the shadowing property if, and only if, it has the finite shadowing property. 
\end{Proposition}
\begin{proof} 

It is straightforward that if $(X,d, f)$ has the shadowing property then it has the finite shadowing property.

To prove the converse, given $\varepsilon>0$, take $\delta>0$ such that any finite $\delta$-pseudo-orbit is $\varepsilon/2$-shadowed by some point. Since $X$ is uniformly locally compact, we can assume, without loss of generality, that $\varepsilon$ is sufficiently small so that the closure of any open ball $B_{\varepsilon}(\x)$ is compact.

Given $(\x_i)_{i\in\N}\in X$ an infinite $\delta$-pseudo-orbit we can consider, for each $k\geq 1$, the finite $\delta$-pseudo-orbit $(\x_i)_{0\leq i\leq k}$. Let $\z_k\in X$ be the point that $\varepsilon/2$-shadows $(\x_i)_{0\leq i\leq k}$. Note that for each $k$ the point $\z_k$ belongs to $B_{\varepsilon/2}(\x_0)$ and, since $\overline{B_{\varepsilon/2}(\x_0)}$ is compact, there exists $\z\in B_{\varepsilon}(\x_0)$ which is an accumulation point of $(\z_k)_{k\geq 1}$. Since $f^i(\z_k)\in B_{\varepsilon/2}(\x_i)$ for all $k\geq 1$ and $0\leq i\leq k$, $f^i$ is continuous for all $i\geq 1$, and $\z$ is accumulation point of $(\z_k)_{k\geq 1}$, it follows that $f^i(\z)\in \overline{B_{\varepsilon/2}(\x_i)}\subset B_{\varepsilon}(\x_i)$ for all $i\geq 0$, that is, $\z$ $\varepsilon$-shadows $(\x_i)_{i\in \N}$.

\end{proof}
\begin{Corollary}\label{Cor:Rn}
The notions of finite shadowing and shadowing coincide in $\R ^n$ with the usual metric. 
\end{Corollary}
\begin{proof}
This simply follows from applying Proposition~\ref{UnifLCThm} to the uniformly locally compact metric space $\R^n$.
\end{proof}

The following example shows that Proposition~\ref{UnifLCThm} is sharp.

\begin{Example}\label{FinShadowEx}
There exists a locally compact dynamical system $(X,d,f)$, with $X \subseteq \R$, that has the finite shadowing property but not the shadowing property. 
\end{Example}
\noindent
{\em Construction:}  Our space $X$ will be a subset of 
	\[ Y = \left \{ n+ \frac{1}{k}: k, n \in \N, k \ge 2 \right \}.
	\]
	It is clear that endowed with the metric of $\R$, any  $X \subseteq Y$ is locally compact as  each point of $X$ is an isolated point of $X$.
	
Let us  now proceed with the construction of our space $X$. 

We first enumerate the set of all finite sequences of  positive integers as $\{s_k\}_{k \in \N}$, i.e., $s_k = (n_{k,0}, \ldots , {n_{k, l(k)}})$ where $l(k)\in \N$  and $n_{k,i}$ is a positive integer. Moreover, we require that every finite sequence of positive integers occurs infinitely often in $\{s_k\}$, i.e., given $(m_0, \ldots, m_j)$, there are infinitely many $k$'s such that $s_k = (m_0, \ldots, m_j)$. 
	
We now construct $A_k$, a finite subset of $Y$, based on $s_k$. More specifically,  let $\{A_k\}_{k \in \N}$ be a sequence of subsets of $Y$ such that the following properties hold.

\begin{enumerate}
	\item $A_k \subseteq \bigcup_{i=0}^{l(k)} (i, i+\frac{1}{k+2})$.
	\item For each $k$ and $0 \le i \le l(k)$ we have that the cardinality of $A_k \cap (i, i+\frac{1}{k+2})$ is $n_{k,i}$.
	\item if $k \neq k'$, then $A_k \cap A_{k'} = \emptyset$.
\end{enumerate}
Finally, we define our space $X = \bigcup _{k \in \N} A_k$. 

We next define a map $f:X \rightarrow X$. It will have the property that $f(A_k) \subseteq A_k$ for all $k$. If $\x \in A_k$ is not the largest element of $A_k$, then $f(\x)$ is the smallest element of $A_k$ greater than $\x$. If $\x$ is the largest element of $A_k$, then $f(\x) = \x$. Clearly, $f$ is a well-defined continuous function on $X$. Moreover, the orbit of every $\x \in X$ is bounded under $f$. 

We now observe that $f$ does not have the shadowing property. For this it suffices to construct, for all $\delta >0$,  an unbounded $\delta$-pseudo-orbit.  Indeed, let $\delta >0$. Choose $N \in \N$ such that $\frac{1}{N} < \delta$. For each $i \in \N$ choose $k_i > N$ so that $l(k_i) > i$. Let $\x_i$ be the largest element of $A_{k_i} \cap  (i, i+\frac{1}{k_i+2})$. Note that for all $i \in \N$, $\x_i \in (i, i+\frac{1}{k_i+2}) \subseteq (i, i+\frac{1}{N})$. Moreover, as $l(k_i) > i$, we have that \[f(\x_i) \in (i+1, i+1+ \frac{1}{k_i+2})
\subseteq (i+1, i+1+ \frac{1}{N}).\] Hence, we have that $\x_{i+1}$ and $f(\x_i)$ are in $(i+1, i+1+ \frac{1}{N})$, implying that $\{\x_i\}_{i \in \N} $ is an unbounded $\delta$-pseudo-orbit.

We next show that $f$ has the finite shadowing property. Let $\varepsilon >0$. Let $N >2$ be  a positive integer such that $\frac{1}{N} < \varepsilon$. As each point of the finite set $\cup_{i=0}^N A_i$ is an isolated point of $X$, we may choose $0 < \delta < \frac{1}{N}$ sufficiently small so that a $\delta$ interval around any point of $\cup_{i=0}^N A_i$ is a singleton set. Let $\{\x_i\}_{i=0}^{j}$ be a finite $\delta$-pseudo-trajectory. By our choice of $\delta$, we have that the entire pseudo-trajectory $\{\x_i\}_{i=0}^{j}$ is a subset of $\cup_{i=0}^N A_i$, or the entire pseudo-trajectory $\{\x_i\}_{i=0}^{j}$ is a subset of $\cup_{i=N+1}^ {\infty} A_i$. In the former case, the pseudo-trajectory $\{\x_i\}_{i=0}^{j}$ is actually a trajectory and we are done. In the latter case, we proceed as follows. We first observe that by our construction of space $X$, map $f$ and the fact that $\delta < 1/2$, we have that if $\x_i \in (u, u+1)$ (for some $u\in \N$) then $\x_{i+1} \in (u, u+1)$ or $\x_{i+1} \in (u+1, u+2)$. Hence, the pseudo-trajectory $\{\x_i\}_{i=0}^{j}$ starts in some interval of the form $(u,u+1)$, then eventually proceeds to the interval $(u+1, u+2)$ and so on, until it terminates. Keeping this structure in mind, we proceed to approximate $\{\x_i\}_{i=0}^{j}$ by a real trajectory.
Let $u \in \N$ be such that $\x_0 \in (u, u +1)$ and let $v \in \N$ be such that $\x_j \in  (v, v+1)$. For each $u \le t \le v$, let $n_t$ be the cardinality of elements of $\{\x_i\}_{i=0}^{j}$ in the interval $(t,t+1)$. Hence the pseudo-trajectory $\{\x_i\}_{i=0}^{j}$ starts in $(u, u+1)$, and stay there $n_u$ times. Then it moves to $(u+1, u+2)$ and spends $n_{u+1}$ times there, etc. Now choose $k > N$ such that $s_k = (n_{k,0}, \ldots, n_{k,l(k)})$ has the property that $n_{k,t} = n_t$ for all $u \le t \le v$. Let $\y$ be the smallest element of $A_k$ in $ (u,u+1)$. We claim that $\y$ $\varepsilon$-shadows $\{\x_i\}_{i=0}^{j}$. Indeed, as $\{\x_i\}_{i=0}^{j}$ and $\{f^i(\y)\}_{i\in \N}$ are subsets of $ \bigcup _{i=N+1}^{\infty}A_i$, we have that the intersection of either one with $(t,t+1)$, $ u \le t \le v$, is actually contained in $(t, t+1/N)$. Moreover, both sequences start in $ (u,u+1/N)$ and spend the same amount of time in each interval before moving to the next one. As $\frac{1}{N} < \varepsilon$, we have that $\y$ $\varepsilon$-shadows $\{\x_i\}_{i=0}^{j}$, completing the proof.
\qed\\

In the case of a finite alphabet, it is a classical result of Walters that a subshift has the shadowing property if, and only if, it is a SFT, \cite[Theorem 3.33]{kurka}. Below we prove an analogous result for infinite alphabets, with the appropriate modifications. 

\begin{Proposition}\label{shadow-finiteshadow-porder} Let $A$ be a countable alphabet and $X\subseteq A^\N$ be a shift space. Then the following statements are equivalent:

\begin{enumerate}
\item\label{prop:SFT-finite_order} $X$ is a shift of finite order;
\item\label{prop:shadow} $(X, d_\Pi, \sigma)$ has the shadowing property;
\item\label{prop:finite_shadow} $(X, d_\Pi, \sigma)$ has the finite shadowing property.
\end{enumerate}

\end{Proposition}

\begin{proof} \phantom\\

$(i) \Longrightarrow (ii)$
Suppose that $X$ is a shift of order $p$. Given $\varepsilon>0$, take $\delta:=(k+1)^{-1}$, where $k\geq p-1$ is an integer such that $(k+1)^{-1}<\varepsilon$. Now, let $(\x_i)_{i\in\N}$ be any infinite $\delta$-pseudo-orbit. Since $d_\Pi\big(\sigma(\x_i),\x_{i+1}\big)<\delta=(k+1)^{-1}$, it follows that $x_{i,1}x_{i,2}...x_{i,k+1}=x_{i+1,0}x_{i+1,1}...x_{i+1,k}$, and recursively we get that $x_{i,\ell}=x_{i+j,\ell-j}$ for all $i\in\N$, $1\leq \ell\leq k+1$ and $0\leq j\leq \ell$. In particular, this implies that \begin{equation}\label{eq:finite_order-shadow}x_{i,1}x_{i,2}...x_{i,k+1}=x_{i+1,0}x_{i+2,0}...x_{i+k+1,0}\end{equation} for all $i\in\N$.  Let $\z=(x_{n,0})_{n\in\N}$. Equation  \eqref{eq:finite_order-shadow}, $k +1 \ge p$, and that $X$ is of order $p$, implies that $\z$ lies in $X$.  Finally, note that \eqref{eq:finite_order-shadow} also implies that $d_\Pi(\sigma^i(\z),\x_i)<(k+1)^{-1}<\varepsilon$.\\

$(ii) \Longrightarrow (iii)$ It is direct.\\

$(iii) \Longrightarrow (i)$ To prove that $X$ is a shift of finite order, we use Proposition~\ref{prop:p-order_characterization}. More precisely, we will show that there exists $p\geq 0$ such that for all $\mathbf{u},\mathbf{v},\mathbf{w}\in \B(X)$ with $|\mathbf{v}|= p-1$, and $\mathbf{uv},\mathbf{vw}\in \B(X)$, we have that $\mathbf{uvw}\in \B(X)$.

Take $0<\varepsilon<1$, and let $\delta>0$ be such that any finite $\delta$-pseudo-orbit is $\varepsilon$-shadowed by some point. Take $p\in\N$ so that $p^{-1}<\delta$. Suppose 
$\mathbf{u}=u_0u_1...u_m,\mathbf{v}=v_{m+1}v_{m+2}...v_{m+p-1},\mathbf{w}=w_{m+p}w_{m+p+1}...w_{m+p+n}\in \B(X)$ are such that $\mathbf{uv},\mathbf{vw}\in \B(X)$ (if $p=1$ consider $\mathbf{v}=\epsilon$, the empty letter such that $\mathbf{uv}=\mathbf{u}$ and $\mathbf{vw}=\mathbf{w}$). Consider $(\x_i)_{0\leq i\leq m+p+n}$ defined as follows: $\x_0$ is any sequence of $X$ starting with the word $\mathbf{uv}$; for $i=1,..., m$ define $\x_i:=\sigma^i(\x_0)$; $\x_{m+1}$ is any sequence of $X$ starting with the word $\mathbf{vw}$; for $i=m+2,..., m+p+n$ define $\x_i:=\sigma^{i-m-1}(\x_{m+1})$. It follows that $(\x_i)_{0\leq i\leq m+p+n}$ is a finite $\delta$-pseudo-orbit since for all $i=0,..., m+p+n-1$ we have that $\sigma(\x_i)$ and $\x_{i+1}$ coincide at least in the first $p-1$ entries, and therefore $d_\Pi\big(\sigma(\x_i),\x_{i+1}\big)\leq p^{-1}<\delta$. Now, let $\z\in X$ be a point that $\varepsilon$-shadows $(\x_i)_{0\leq i\leq m+p+n}$. Since $d_\Pi\big(\sigma^i(\z),\x_i\big)<\varepsilon<1$ for all $i=0,...,m+p+n$, it follows that $z_i=x_{i,0}$ for all $i=0,...,m+p+n$, which means that the sequence $z$ starts with the word $\mathbf{uvw}$, and then $\mathbf{uvw}\in \B(X)$.
\end{proof}

\section{Defining Sequences, Ultrametric Spaces and $P$-adics}
In this section, we define three notions that are essential for our main results. Our definitions go from the most general to specific. The notion of a tame defining sequence is new as far as we know. Of course, ultrametric spaces and $p$-adics are rather well known. 

\subsection{Defining sequences and ultrametric spaces} 

Recall that a {\em partition} $\mathcal{U}$ of a space $X$ is a covering by pairwise disjoint, nonempty clopen sets. Since our spaces are always separable, the partitions we work with are always countable. For $\x\in X$, we denote by $\mathcal U [\x]$ the carrier of $\x$, that is, the unique element of $\mathcal U$ which contains $\x$.

\begin{Definition}\label{nintendoDS}
 A \emph{defining sequence} of a space $X$ is a sequence $\mathcal A = \{\m{U}{n}\}_{n \in \N}$ of partitions of $X$ such that:
  \begin{enumerate}
      \item\label{refin} $\m{U}{n+1}$ is a refinement of $\m{U}{n}$, i.e., each element of $\m{U}{n+1}$ is a subset of some (necessarily unique) element of $\m{U}{n}$.
      \item\label{blob} The collection $\{U: U \in \m{U}{n}, \text{ for some } n \in \N\}$ is a basis for the topology in $X$.
           \end{enumerate}
Furthermore, a defining sequence $\mathcal A$ is called {\em complete} when it satisfies the following condition
\begin{itemize}
\item If $U_n \in \m{U}{n}$ with $U_{n+1} \subseteq U_n$ for all $n\in \N$, then $\bigcap_n U_n$ is nonempty.
\end{itemize}
\end{Definition}

\begin{Remark}
Notice that if $\mathcal A$ is a complete defining sequence and $\{U_n\}$ is a sequence as above, then $\bigcap_n U_n$ has exactly one element. Indeed, let $\x_1\in \bigcap_n U_n$, $\x_2\neq \x_1$, and take $V$ as an open set that contains $\x_1$ but does not contain $\x_2$. Then, by item~\ref{blob} in Definition~\ref{nintendoDS}, there exists an $N\in \N$ such that $\x_1\in U_N \subseteq V$. Hence $\x_2$ does not belong to $\bigcap_n U_n$.
\end{Remark}

We will be interested in spaces that admit (complete) defining sequences. In the next result we show that this is the case for (Polish) zero-dimensional spaces. 

\begin{Proposition}
A zero-dimensional space $X$ admits a defining sequence. Furthermore, if $X$ is Polish then it admits a complete defining sequence.
\end{Proposition}
\begin{proof}
Choose a metric $d$ for $X$ and, if $X$ is Polish, then choose $d$ complete. Let $\mathcal{U}_0=\{X\}$. We will proceed by induction.

Suppose that $\mathcal{U}_k$ has been defined for each $k\leq n$. For each $\x \in X$, choose a clopen neighborhood contained in $\mathcal{U}_n[\x]$ and with diameter less than $\frac{1}{n+1}$. The collection of all such neighborhoods cover $X$. Let $V_0,V_1,\ldots$ be a countable subcover and define $U_k=V_k\setminus \bigcup_{i<k} V_i$. Let $\mathcal{U}_{n+1}=\{U_k:k\in \N, U_k\neq \emptyset\} $. Notice that $\{\mathcal{U}_{n}\}_{n\in\N}$ is a defining sequence. Indeed, if $\varepsilon>0$ and $\x\in X $ then $\mathcal{U}_n[\x]\subset B_\varepsilon(\x)$ when $\frac{1}{n+1}<\varepsilon$ and hence item~\ref{blob} in Definition~\ref{nintendoDS} is satisfied.

If the metric is complete, the condition required for $\{\mathcal{U}_{n}\}_{n\in\N}$ to be a complete defining sequence follows from the Cantor Intersection Theorem, since the $d$ diameters of the nonempty closed sets $U_n$ converges to zero.

\end{proof}

As it happens, the converse of the second statement of the above result is true, that is, if a zero-dimensional space admits a complete defining sequence then it is Polish. To prove this we need to introduce a few concepts first, starting with ultrametric spaces.

\begin{Definition}\label{jeronimo}
An \emph{ultrametric space}
is a metric space $(X,d)$ where the ultrametric inequality holds, that is, \[d(\x,\z)\leq \max \{d(\x,\y),d(\y,\z)\}\] for all $\x,\y,\z\in X$.
\end{Definition}
We notice that a metric space $(X,d)$ is  an ultrametric space if, and only if, for all $\varepsilon>0$ the set $\{(\x,\y):d(\x,\y)<\varepsilon\}$ is an open (and hence clopen) equivalence relation. This implies the following properties which we will use in sequel.
\begin{enumerate}[label=(\subscript{UM}{{\arabic*}})]
    \item\label{patinete} All balls of strictly positive radius are both open and closed in the induced topology.
    \item\label{triciclo} If $d(\x,\y)<r$ then $B(\x,r)=B(\y,r)$, where $B(\z,r)$ denotes the ball centered at $z$ of radius $r$.
    \item\label{bike} Either the intersection of two balls is empty or one is contained in the other. 
    \item\label{hoverboard} The distance between any two distinct balls of radius $r>0$ is $r$ or greater.
\end{enumerate}

\begin{Remark}
Shifts on countable alphabets with the metric of Definition~\ref{defnshifts} are ultrametric spaces. 
\end{Remark}

We next show that defining sequence give rise to a natural ultrametric.
If $\mathcal{U}$ is a partition of $X$ then  \[ E_\mathcal{U} := \{(\x,\y):\x\in X, \y\in \mathcal U[\x]\} = \bigcup_{\x\in X} \mathcal U[\x] \times \mathcal U [\x] \] is a clopen equivalence relation, with $\mathcal U$ as the set of equivalence classes.

Let $\mathcal A=\{\mathcal{U}_n\}$ be a defining sequence of $X$. For $\x,\y\in X$ define the ultrametric associated with $\mathcal{A}$ by 
\begin{equation}\label{polvo}
u_\mathcal A (\x,\y):= \frac{1}{1+j}, \textit{ where } j = \inf_n \left\{ \mathcal{U}_n[\x] \neq \mathcal{U}_n[\y]\right\} .
\end{equation}
We then have the following.

\begin{Proposition}\label{toalha} Let $\mathcal A=\{\mathcal{U}_n\}$ be a defining sequence of $X$ and $u_\mathcal A$ as above. Then $u_\mathcal A$ is an admissible ultrametric on $X$ which is complete if, and only if, $\mathcal A$ is complete.
\end{Proposition}
\begin{proof}
Notice that, for any $\varepsilon$ with $\frac{1}{n+1}<\varepsilon\leq \frac{1}{n}$ the set $\{(\x,\y):u_{\mathcal{A}}(\x,\y)<\varepsilon\}=E_{\mathcal{U}_{n-1}}$. As this is an equivalence relation, $u_\mathcal A$ is an ultrametric and, since $B_\varepsilon(\x)=\mathcal{U}_{n-1}[\x]$, it follows that $u_\mathcal A$ is an admissible metric.

If $u_\mathcal A$ is complete then completeness of $\mathcal A$ follows from the Cantor Intersection Theorem. Conversely, if $\{\x_n\}$ is a $u_\mathcal A$ Cauchy sequence, then for every $n$ there exists $N$ such that $U_n=\mathcal{U}_n[\x_N]=\mathcal{U}_n[\x_{N+k}]$ for all $k\in \N$. Clearly, $U_{n+1}\subset U_n$ and if $\x$ is the point in the intersection given by the completeness of $\mathcal A$ then $\{\x_n\}$ converges to $x$.
\end{proof}

\begin{Corollary}
A zero-dimensional space admits a complete defining sequence if, and only if, it is Polish.
\end{Corollary}

For our main theorem, we will need a ``tame" defining sequence. More precisely, we have the following definition.

\begin{Definition}\label{tameDS}
  Let $(X,d)$ be a metric space and $\{\m{U}{n}\}_{n \in \N}$ be a defining sequence of $X$. For all $n\in \N$, let \[S_n=\sup\{\text{diam}(O): O \in \m{U}{n}\}\] 
where $\text{diam}(O)$ stands for the diameter of the set $O$. We say that $\{\m{U}{n}\}_{n \in \N}$ equipped with the metric $d$ is a \emph{tame} defining sequence of $X$ if $S_n \rightarrow 0$ as $n\rightarrow\infty$ and, for all $n\in \N$, there exists $\rho_n >0$ such that if $O_1, O_2$ are distinct element of $\m{U}{n}$ and $\x_i \in O_i$, then $d(\x_1,\x_2) \ge \rho_n$. 
For such $\rho_n$, we say that \emph{$\m{U}{n}$ is $\rho_n$-separated. }
\end{Definition}


Notice that while a defining sequence is a topological concept, tameness is a metric space concept. Clearly, replacing the metric by a uniformly equivalent metric preserves tameness. More precisely, we have the following.

\begin{Proposition}\label{berries}
Any defining sequence $\mathcal A$ on space $X$ equipped with the associated ultrametric $u_\mathcal A$ is tame. Furthermore, a defining sequence  $\mathcal A$ equipped with an admissible metric $d$ is tame  if, and only if, $d$ is uniformly equivalent to $u_\mathcal A$.
\end{Proposition}

\begin{Corollary}\label{platano}
A defining sequence $\mathcal A$ is tame with respect to admissible metrics $d$ and $d_1$ if, and only if, $d$ and $d_1$ are uniformly equivalent. 
\end{Corollary}

Ultrametric spaces have tame defining sequences that are easy to describe, as we see below. 

\begin{Proposition}\label{cabra}
Let $(X,d)$ be an ultrametric space and let $\mathcal{U}_n = \{B(\x,\frac{1}{n}):\x\in X\}$, for $ n\in \N_+$, and $\mathcal{U}_0=\{X\}$. Then the sequence $\mathcal A=\{\mathcal{U}_n\}_{n\in \N}$ is a tame defining sequence of $X$, which is complete if, and only if, the ultrametric $u_\mathcal A$ is complete. 
\end{Proposition}
\begin{proof}

It is clear that $\mathcal A=\{\mathcal{U}_n\}_{n\in \N}$ satisfies items \ref{refin} and \ref{blob} of Definition~\ref{nintendoDS}, so it is a defining sequence. To check that it is tame, we notice that the diameter of each ball in $\mathcal{U}_n$ is $\frac{1}{n}$ (so $S_n\rightarrow 0$) and that each $\mathcal{U}_n$ is $\frac{1}{n}$-separated by Property~\ref{hoverboard}. For the completeness statement, notice that by Corollary~\ref{platano} the metric $d$ is uniformly equivalent to the metric $u_\mathcal A$, and hence the result follows from Proposition~\ref{toalha}.
\end{proof}

\subsection{$P$-adics}
We finish this section by recalling the construction of the $p$-adic integers  and the $p$-adic rationals. The standard norms on these spaces generate  metrics which  are ultrametrics. As such, they admit tame defining sequences.  For the sake of concreteness, we also give explicit descriptions of these tame defining sequences. For general information on $p$-adics, we refer the reader to \cite{Robert}.

\begin{Definition}\label{defnpadics}
Let $p$ be a prime number and let $A= \{0, \ldots, p-1\}$, i.e., the field of integers modulo $p$. Formally, we define $\Z_p$ and $\Q_p$ as follows:

\[\Z_p := \left \{\sum_{i \in \Z } a_i p^i : a_i \in A , a_i =0\  \forall  \  i < 0 \right \} \]
\[\Q_p := \left \{\sum_{i \in \Z } a_i p^i : a_i \in A , \exists l \in \Z \text{ such that } a_i =0 \  \forall \   i < l \right\} .\]

Elements in $\Z_p$ and $\Q_p$ are summed pointwise modulo $p$ with a carryover. Multiplication of an element of $\Q_p$ with a scalar in $A$ is defined pointwise with a carryover. Using addition and multiplication by a scalar, multiplication in $\Z_p$ and $\Q_p$ is defined in a natural way. Equipped with these algebraic operations, we have that  $\Z_p$ is a ring and $\Q_p$ is a field.
\end{Definition}
We now recall $p$-valuation and the metric it induces.
\begin{Definition}\label{normpadic}
Let  $\x =\sum_{i \in \Z}a_i p^i \in \Q_p$. If $a_i = 0$ for all $i$, then $\|\x\|_p=0$. Otherwise, $\|\x\|_p = p^{-l}$, where $l$ is the least integer such that $a_l\neq0$.

Then, $\| \ \|_p$ induces natural metrics on $\Z_p$ and $\Q_p$ by 
$d(\x,\y) = \|\x - \y \|_p.$
It is easily verified that they are complete ultrametrics. Moreover, equipped with this metric, $\Z_p$ is a compact completion of $\Z$  homeomorphic to the Cantor space, whereas $\Q_p$ is a locally compact completion of $\Q$.
\end{Definition}
Next we give an explicit description of the tame defining sequences formed by the ultrametric on $\Z_p$ and $\Q_p$. 

We denote the set of all words on $A$ of length $n$ by $A^n$, that is, \[A^n=\{\sigma_0 \sigma_1 \ldots \sigma_{n-1}: \sigma_i \in A\},\] 
recalling $A^0=\{\varepsilon\}$ where $\varepsilon$ is the empty word.  For $\sigma \in A^n$, say $\sigma=\sigma_0 \ldots \sigma_{n-1}$, let \[ \Z_p(\sigma):= \left \{\sum_{j\in \Z} a_j p^j \in \Z_p : a_{j} = \sigma_{j} \text{ for } 0\leq j \leq n-1 \right \},\]
and let $ \Z_p(\varepsilon):=\Z_p$. Then, letting 
\begin{align*}
   \mathcal{U}_n&=\{\Z_p(\sigma):\sigma \in A^n\}\\
\mathcal{V}_n &= \{ \Z_p(\sigma) p^j: \sigma\in A^{n-j},\  \sigma_0 \neq 0, \  j\in \Z, \ j\leq n \},
\end{align*}
we have that $\{\mathcal{U}_n\}$ and $\{\mathcal{V}_n\}$ are tame defining sequences for $\Z_p$ and $\Q_p$, respectively. Indeed such is the case, as ${\mathcal U}_n$ and ${\mathcal V}_n$ consist of open balls of radius $p^{-n+1}$ in $\Z_p$ and $\Q_p$, respectively.
\section{Inverse Limits, Shadowing and 
Shifts of Finite Order }\label{ovelha}

We begin by introducing the basic terminology of inverse limits.

\subsection{Inverse Limits and Shadowing}
\begin{Definition}\label{defn:invlim}  A {\em inverse system} consists of a sequence  $(g_m,X_m)$ where, for each $m\in \N$, $X_m$ is a space and $g_m:X_{m+1}\to X_{m}$ is a continuous map. 
	The \emph{inverse limit of $(g_m,X_m)$} is the space
	\[\invlim{g_m}{X_m}:=\left\{\seq{\x_m}_{m\in\N}\in\prod X_m:  \ \x_m=g_m(\x_{m+1}) \ \forall m \in \N  \right \},\]
	with the subspace topology inherited from the product topology on $\prod X_m$.
	The maps $g_m$'s are called \emph{bonding maps}.
	
\end{Definition}

\begin{Remark}
Notice that the inverse limit $\invlim{g_m}{X_m}$ is a closed subspace of the product $\prod X_m$. If each space $X_m$ is a complete metric space, then the inverse limit $\invlim{g_m}{X_m}$ is a complete metric space. Furthermore, if one equips $\prod X_m$ with the metric given in Equation~(\ref{productmetric}), then the inverse limit of ultrametric spaces is an ultrametric space.
\end{Remark}

The following condition is essential to guarantee that the inverse limit of spaces with the shadowing property has the shadowing property.

\begin{Definition}
	The inverse  system $(g_m,X_m)$ satisfies the \emph{Mittag-Leffler Condition} if for every $N\in \N$ there is a $k> N$ such that for all $i\geq k$ the following holds \[g_N\circ \ldots \circ g_k(X_{k+1}) = g_N \circ \ldots g_i(X_{i+1}).\]
\end{Definition}

\begin{Remark}\label{ML} Suppose that  $(g_m,X_m)$ satisfies the Mittag-Leffler Condition. Let $Y_n=\bigcap_{k=1}^{\infty} g_{n} \circ \ldots \circ g_{n+k}(X_{n+k+1})$, $n\in \N$. Then, the inverse system $(g_m,Y_m)$ has surjective bonding maps and the same inverse limit as the original sequence $(g_m,X_m)$.
\end{Remark}


We next define inverse limit dynamical system which is a subsystem of the product dynamical system. Let $(X_m, d_m, f_m)$ be a sequence of dynamical systems. Recall that the product map $ \prod_m f_m : \prod X_m \rightarrow \prod X_m$ is defined as \[ \prod_m f_m \left((\x_m)\right) = \left(f(\x_m)\right).\]
The {\em product dynamical system} is defined as $(\prod X_m,d_\Pi, \prod_m f_m)$.

\begin{Definition}\label{paralisado}
	Let $\{(X_m,d_m, f_m)\}$ be a  sequence of dynamical systems. Assume that  $\{g_m\}$, $g_m:X_{m+1}\to X_{m}$, is a sequence  of continuous bonding maps that satisfies the following property:  
	\[f_m\circ g_m=g_m\circ f_{m+1}.\]
	The  \emph{inverse limit dynamical system of  $(g_m,(X_m,f_m))$}, or \emph{ the inverse limit} for short, is the dynamical system  $(\invlim{g_m}{X_m}, d_\Pi, (f_m)^*) $, where $(f_m)^*$ is the restriction of $\prod_m f_m$ to $\invlim{g_m}{X_m}$. 
\end{Definition}

Notice that $\prod_m f_m$ is continuous and $(\invlim{g_m}{X_m},d_\Pi, (f_m)^*) $ is a continuous dynamical system. Furthermore, if each $f_m$ is uniformly continuous then $(f_m)^*$ is also uniformly continuous. 

A straightforward, but useful, result that we will need in the sequel is the following.

 \begin{Lemma}\label{esteaqui}
 Under the hypothesis of Definition~\ref{paralisado}, suppose further that each $X_i$ is equipped with a metric $d_i$ bounded by 1. If $(\x_i)_{i\in\N}$ is a $\delta$-pseudo-orbit in $(\invlim{g_m}{X_m}, d_I, (f_m)^*) $
 then, for every fixed $m\in \N$, the sequence $(\x_{i,m})_{i\in \N}$ is a $(m+1) \delta$-pseudo-orbit  of $f_m$ in $X_m$.
 \end{Lemma}
 \begin{proof}
 Just notice that
 \[ \frac{d_m \left( f_m(\x_{i,m}), \x_{i+1,m} \right)}{m+1}
 \leq \max_j\left\{ \frac{d_j \left( f_j(\x_{i,j}), \x_{i+1,j} \right)}{j+1}\right\} = 
 d_\Pi\left((f_l)^*\left((\x_i)\right), \x_{i+1}\right) < \delta,\] for all $i$. Hence $(\x_{i,m})_{i\in \N}$ is a $(m+1) \delta$-pseudo-orbit in $X_m$ for every $m\in \N$.
 
 \end{proof}

 We now prove the main result of this subsection.

\begin{Theorem}\label{caipirinha}
Let $\{(X_m, d_m, f_m)\}$ be a sequence of dynamical systems such that each $X_i$ has a metric $d_i$ bounded by 1. Let $\{g_m\}$, with $g_m:X_{m+1}\to X_m$, be a sequence of uniformly continuous bonding maps. Suppose that the inverse system $(g_m,X_m)$ satisfies the Mittag-Leffler Condition. 
\begin{enumerate}
    \item  If $\{(X_m, d_m, f_m)\}$ has the shadowing property for every $m$, then the inverse limit dynamical system $(\invlim{g_m}{X_m},d_\Pi, (f_m)^*) $ has the shadowing property.
    \item If $\{(X_m, d_m, f_m)\}$ has the finite shadowing property for every $m$, then $(\invlim{g_m}{X_m},d_\Pi, (f_m)^*) $ has the finite shadowing property.
\end{enumerate}
 \end{Theorem}
 \begin{proof}
We prove only (i) as the proof of (ii) is analogous. In light of Remark~\ref{ML}, it suffices to show the result when $g_m$ is surjective for all $m$.

Given $\varepsilon>0$, choose $k\in \N$ such that $\frac{1}{k} < \varepsilon$. 

Let $\varepsilon'<\varepsilon$ be such that if two points in $X_k$ are $\varepsilon'$ close (that is, their distance is less than $\varepsilon'$) then their images under $g_i \circ \ldots  \circ g_{k-1}$, for $0\leq i \leq k-1$, is $\varepsilon$ close. Such $\varepsilon'$ exists from the uniform continuity of $g_i$'s.

Let $\delta_k>0$ be such that every $\delta_k$-pseudo-orbit in $X_k$ is $\varepsilon'$-shadowed. Let $\delta = \frac{\delta_k}{k+1}$. We will show that every $\delta$-pseudo-orbit in $(\invlim{g_m}{X_m},d_\Pi, (f_m)^*) $ is $\varepsilon$-shadowed.

Let $(\x_i)$ be a $\delta$-pseudo-orbit in $(\invlim{g_m}{X_m}, d_\Pi,(f_m)^*) $. By Lemma~\ref{esteaqui} we have that $(\x_{i,k})_{i=1}^\infty$ is a $(k+1) \delta$-pseudo-orbit in $X_k$. Hence, it is a $\delta_k$-pseudo-orbit in $X_k$. By our choice of $\delta_k$, we  may choose $\mathbf{t}_k$ in $X_k$ that $\varepsilon'$-shadows $(\x_{i,k})_{i=1}^\infty$, i.e., $d_k(f_k^i(\mathbf{t}_k),\x_{i,k})<\varepsilon'$ for all $i$. 
As $g_m$'s are surjective, we may choose $\mathbf{t}\in (\invlim{g_m}{X_m},(f_m)^*) $ such that the projection of $\mathbf{t}$ onto the $k^{th}$-coordinate is $\mathbf{t}_k$.

We now show that $\mathbf{t}$ is a point that $\varepsilon$-shadows $(\x_i)$. Indeed,

\begin{align*} d_\Pi\left( \left((f_j)^*\right)^i(\mathbf{t}), \x_i \right) & = \max_j\left\{\frac{d_j ( f_j^i(\mathbf{t}_j),\x_{i,j})}{j+1}\right\}\\
 & = \max\left\{\max_{j=0,\ldots,k-1}\left\{\frac{d_j ( f_j^i(\mathbf{t}_j),\x_{i,j})}{j+1}\right\}, \max_{j\geq k}\left\{\frac{d_j ( f_j^i(\mathbf{t}_j),\x_{i,j})}{j+1}\right\}\right\}
 \\
  & \leq \max\left\{\max_{j=0,\ldots,k-1}\left\{\frac{d_j ( f_j^i(\mathbf{t}_j),\x_{i,j})}{j+1}\right\},\frac{1}{k+1}\right\}
 \\
 & = \max\left\{\max_{j=0,\ldots,k-1}\left\{\frac{d_j (f_j^i(g_j \circ \ldots \circ g_{k-1}(\mathbf{t}_k)),\x_{i,j})}{j+1}\right\},\frac{1}{k+1}\right\}
 \\
 & = \max\left\{\max_{j=0,\ldots,k-1}\left\{\frac{d_j ( g_j \circ \ldots \circ g_{k-1}(f_k^i(\mathbf{t}_k)),\x_{i,j})}{j+1}\right\},\frac{1}{k+1}\right\}
 \\
 & = \max\left\{\max_{j=0,\ldots,k-1}\left\{\frac{d_j ( g_j \circ \ldots \circ g_{k-1}(f_k^i(\mathbf{t}_k)),g_j\ldots g_{k-1}(\x_{i,k}))}{j+1}\right\},\frac{1}{k+1}\right\}
 \\
 & \leq \varepsilon.
\end{align*}
where the last line follows from the previous one by our choice of $\mathbf{t}_k$ and the uniform continuity of $g_j \circ \ldots \circ g_{k-1}$.
This concludes the proof of the shadowing property of  $(\invlim{g_m}{X_m},d_\Pi, (f_m)^*) $ as desired.
\end{proof}

\subsection{Shadowing and Shifts of Finite Order}\label{Shadowing and Shifts of Finite Order}
Our aim now will be to prove a converse of the above result. More precisely we will show that a dynamical system $(X,d,f)$ with a complete tame defining sequence and the finite shadowing property is conjugate to an inverse limit of 1-step shift spaces over countable alphabets, and this conjugacy can be made uniform when the map $f$ is uniformly continuous. In this section we will also characterize the relation between the finite shadowing and shadowing property.

Let $\m{U}{}$ be a partition of $X$.
Then, we define
\[\m{O}{}(\m{U}{},f) = \left  \{(O_i)_{i\in\N} \in \m{U}{}^{\N}: \forall \ k \in \N, \ \exists \ \x \in X \textit{ s.t. } f^i(\x) \in O_i, \  0 \le i  \le k  \right \}, \]
and 
\[\m{PO}{}(\m{U}{},f) = \left  \{(O_i)_{i\in\N} \in \m{U}{}^{\N}: \forall \ i \in \N,  f(O_i)\cap O_{i+1}\neq\emptyset\right \}. \]
When we are working with a single function $f$, we will often suppress $f$ and write $\m{O}{}(\m{U}{})$
and $\m{PO}{}(\m{U}{})$, instead of $\m{O}{}(\m{U}{},f)$
and $\m{PO}{}(\m{U}{},f)$, respectively.

We endow $\m{U}{}$ with the discrete metric (so the distance between any two distinct points is one) and $\m{U}{}^{\N}$ with the product topology, which is generated by the metric given in Equation~(\ref{productmetric}).

We now prove a sequence of auxiliary results which lead to the proof of the main theorem.

\begin{Lemma}\label{basicdeflem}
Let $\m{U}{}$ be a  partition of $X$ and $f:X\rightarrow X$ be a continuous function. Then, the following hold.
\begin{enumerate}
    \item $\m{O}{}(\m{U}{}) \subseteq \m{PO}{}(\m{U}{})$
    \item  $\m{O}{}(\m{U}{})$ is a subshift of $\m{U}{}^\N$;
    \item $\m{PO}{}(\m{U}{})$ is a 1-step subshift of $\m{U}{}^\N$.
\end{enumerate}
\end{Lemma}

\begin{proof}\ 
\begin{enumerate}
\item This follows directly from the definitions.

\item It is straightforward that $\m{O}{}(\m{U}{})$ is shift invariant. Thus, we just need to check that $\m{O}{}(\m{U}{})$ is closed in  $\m{U}{}^{\N}$.

Let $(O_\ell)_{\ell\in\N}$ be a sequence of points of $\m{O}{}(\m{U}{})$, that is,   $O_\ell=(O_{\ell, i})_{i\in\N}\in \m{O}{}(\m{U}{})$, for each $\ell\in\N$. Suppose that $(O_\ell)_{\ell\in\N}$ converges to some point $\bar O=(\bar O_i)_{i\in\N}\in\m{U}{}^{\N}$. Then, for all $k\in\N$, there exist $N\in\N$ such that for all $\ell\geq N$ we have $O_{\ell,i}=\bar O_i$ for all $0\leq i \leq k$. Thus, for any fixed $\ell\geq N$ we can take $\x\in X$ such that  $f^i(\x) \in O_{\ell,i}=\bar O_i$ for all  $0 \le i  \le k$, which means that $\bar O\in \m{O}{n}$.

\item From its definition, $\m{PO}{}(\m{U}{})$ is the 1-step shift whose set of forbidden words is given by $F=\{OQ: f(O)\cap Q=\emptyset \}$.  \end{enumerate}
\end{proof}
 
Let $\m{U}{}, \m{V}{}$ be partitions of $X$ with $\m{U}{}$ a refinement of $\m{V}{}$. We use $\m{U}{}^{\N} \hookrightarrow \m{V}{}^{\N}$ do denote the function which takes $(U_i)_{i \in \N} \in \m{U}{}^{\N}$ to the unique element $(V_i)_{i \in \N} \in \m{V}{}^{\N}$ such that $U_i \subseteq V_i$ for all $i$. Moreover, if $\m{A}{} \subseteq \m{U}{}^{\N}$ and $\m{B}{} \subseteq \m{V}{}^{\N}$, then
$\m{A}{} \hookrightarrow \m{B}{}$ denotes the same map restricted to $\m{A}{}$ with the implied assumption that the image of $\m{A}{}$ is contained in $\m{
B}{}$. Keeping in mind the metric given in  Equation~(\ref{productmetric}), we have that $\hookrightarrow$ is a uniformly continuous map from $\m{A}{}$ into $\m{B}{}$. In the specific case that $\m{A}{} = \{(U_i)_{i \in \N}\}$ and $\m{B}{} = \{(V_i)_{i \in \N}\}$, we abuse the notation and simply write $ (U_i)_{i \in \N} \hookrightarrow (V_i)_{i \in \N}$.

We now define a  sequence of maps which yields a conjugacy between the map $f$ and the inverse limit representing it.

Let $\m{U}{}$ be partition of $X$. Define
\[\pi_0 (\m{U}{}): \m{O}{}(\m{U}{})  \rightarrow \m{U}{} \ \ \  \textit{by} \ \ \  \pi_0(\m{U}{})((O_i))  = O_0,\]
\[\alpha (\m{U}{}): X \rightarrow \m{U}{} \ \  \textit{ by } \ \ \alpha(\m{U}{}) (\x) = \m{U}{} [\x], \text{ and }\]
\[\theta (\m{U}{}): X \rightarrow \m{O}{}(\m{U}{}) \ \ \  \textit{by} \ \ \  \theta(\m{U}{}) (\x) = (\m{U}{} [f^i(\x)]). \]

\begin{Proposition}\label{partialmaps}
\begin{enumerate} The following statements hold.
    \item $\pi_0 (\m{U}{})$ is $1$-Lipschitz.
    \item $\alpha (\m{U}{})$ is continuous.
    \item If $f$ is  continuous, then $\theta (\m{U}{})$ is continuous.
    \item If $f$ is  uniformly continuous, and $\m{U}{}$ is $\rho$-separated for some $\rho >0$, then $\theta (\m{U}{})$ is uniformly continuous.
    \item $\m{O}{}(\m{U}{})$ is the closure of $\theta (\m{U}{})(X)$.
\end{enumerate}
\end{Proposition}
\begin{proof}
$(i)$ follows from the definitions of $\pi_0 (\m{U}{})$ and $d_\Pi$. $(ii)$ follows from the fact that $\m{U}{}$ is a partition.
$(iii)$ follows from the definitions of  continuity, $\theta (\m{U}{})$ and $d_\Pi$. 

To see $(iv)$, let $\varepsilon >0$. Let $j \in \N$ be such that $1/(j+1) <\varepsilon$. As each $f^i$, $0 \le i \le j$, is uniformly continuous, let $\delta >0$ satisfy, simultaneously, the uniform continuity of each $f^i$, $0 \le i \le j$, with respect to $\rho$. Now, we have that for any $\x_1, \x_2$ in $X$ with distance less than $\delta$, the first $j$ coordinates of 
$\theta (\m{U}{})(\x_1)$ and $\theta (\m{U}{})(\x_2)$ are equal, implying that their distance is less than $\varepsilon$.

Finally, to see $(v)$, let $(O_i) \in \m{O}{}(\m{U}{})$. By the definition of $\m{O}{}(\m{U}{})$, we have that for each $k \in \N$, there exists $\x \in X$ such that $f^i(\x) \in O_i$, $0 \le i \le k$. Hence, the first $k$ coordinates of $\theta (\m{U}{})(\x)$ agree with first $k$ coordinates of $(O_i)$, implying that $(O_i)$ is in the closure of   $\theta (\m{U}{})(X)$. \end{proof}

We next define maps $\alpha,  \theta$,  and $\pi$ as product maps.

Let $\mathcal{A} =\{\m{U}{n}\}_{n \in \N}$ be a defining sequence of $X$ and $(X,d, f)$ be a dynamical system. Define

\[ \alpha = \prod_{n} \alpha (\m{U}{n}) :X \rightarrow \invlim{\hookrightarrow}{\m{U}{n}},
\]
\[ \theta = \prod_{n} \theta (\m{U}{n}) :X \rightarrow \invlim{\hookrightarrow}{\m{O}{}(\m{U}{n})}, \text{ and}
\]
\[ \pi = \prod_{n} \pi_0 (\m{U}{n}) :\invlim{\hookrightarrow}{\m{O}{}(\m{U}{n})} \rightarrow \invlim{\hookrightarrow}{\m{U}{n}}.
\]
As product maps, each of $\alpha$, $\theta$ and $\pi$ is continuous. The following proposition states this and more.
\begin{Proposition}\label{fullmaps}
The following are true.
\begin{enumerate}
    \item The map $\alpha$ is an isometry from $(X,u_{\alpha})$ to $(\invlim{\hookrightarrow}{\m{U}{n}}, d_{\Pi})$. Moreover, $\alpha$ is onto if $u_{\alpha}$ is complete.
    \item The map $\theta$ is continuous. If $f$ is uniformly continuous and $\mathcal{A}$ is tame, then $\theta$ is uniformly continuous from $(X,d)$ to $(\invlim{\hookrightarrow}{\m{O}{}(\m{U}{n})}, d_{\Pi})$.
    \item The map $\pi$ is 1-Lipschitz from $(\invlim{\hookrightarrow}{\m{O}{}(\m{U}{n})}, d_{\Pi})$ to $(\invlim{\hookrightarrow}{\m{U}{n}}, d_{\Pi})$.
    \end{enumerate}
    \begin{proof}
    $(i)$ simply follows from the definitions of metrics $u_{\alpha}$ and $d_{\Pi}$.
    
    For $(ii)$, that $\theta$ is continuous follows from the fact that it is a product of continuous maps. The second part follows from Proposition~\ref{partialmaps} $(iv)$ and the fact that the product of uniformly continuous maps, as remarked after Equation~(\ref{productmetric}), is uniformly continuous.
    
    $(iii)$ follows from the fact that the product of 1-Lipschitz map is 1-Lipschitz.
    \end{proof}
\end{Proposition}

\begin{Lemma}\label{bebavinhonaquarentena} Let $\sigma_n$ be the shift map on $\m{O}{} (\m{U}{n})$. Then,  $\theta \circ f = \sigma^* \circ \theta$, 
where $ \sigma^*  : =(\sigma_n)^*$ is the map on $\invlim{\hookrightarrow}{\m{O}{}(\m{U}{n})}$ given in Definition~\ref{paralisado}.
\end{Lemma}
\begin{proof}

Let $\x\in X$. Recall that $\theta(\x)=(O_{n})_{n\in\N}$, where $O_{n,l}$ is the unique element of $ \m{U}{n}$ that contains $f^l(\x)$, $l \in \N$, that is, $O_{n,l}=\m{U}{n}[f^l(x)]$. Notice that $f^l(f(\x))=f^{l+1}(\x) \in O_{n,l+1} = \sigma_n (O_{n,l}) $ for each $l \in \N$. Hence \[\theta(f(\x)) = \big(\sigma_n(O_{n})\big)_{n\in\N}= \sigma^*(\theta(\x))\] as desired.

\end{proof}
\begin{Proposition}\label{contTheta} The following hold.
\begin{enumerate}
    \item Suppose $\mathcal{A}$ is complete. Then, $\theta$ is a homeomorphism.
    \item Suppose $\mathcal{A}$ is complete and tame.  Then, $f$ is uniformly continuous if, and only if, $\theta$ and $\theta^{-1}$ are uniformly continuous. 
 
\end{enumerate}
\end{Proposition}
\begin{proof}
$(i)$ That $\theta$ is continuous follows from Proposition~\ref{fullmaps} $(ii)$. Therefore, it suffices to find a continuous map $\beta$ such that  $\beta \circ \theta = id_X$ and $\theta \circ\beta =id_{\invlim{\hookrightarrow}{\m{O}{}(\m{U}{n})}}$.

We observe that $\alpha =  \pi \circ \theta $. As $\mathcal{A}$ is complete, by Proposition~\ref{fullmaps} $(i)$, we have that $\alpha$ maps bijectively onto $\invlim{\hookrightarrow}{\m{U}{n}}$ and hence $\alpha^{-1}$ is well-defined on its domain, namely $\invlim{\hookrightarrow}{\m{U}{n}}$. Moreover, $\alpha^{-1}$ is an isometry from $\invlim{\hookrightarrow}{\m{U}{n}}$ onto $X$. Let $\beta = \alpha^{-1} \circ \pi$. As the composition of two continuous maps is continuous, we have that $\beta$ is continuous. In addition, $\beta \circ \theta = (\alpha^{-1} \circ  \pi )\circ \theta  = \alpha^{-1} \circ ( \pi \circ \theta) = \alpha^{-1} \circ \alpha =id_X$.

Finally, let us show that $\theta \circ\beta =id_{\invlim{\hookrightarrow}{\m{O}{}(\m{U}{n})}}$. To this end, let $(O_n)_{n\in \N}\in \invlim{\hookrightarrow}{\m{O}{}(\m{U}{n})}$.  Fix $i\in \N$. For every $n\in \N$, since $O_n \in \m{O}{}(\m{U}{n})$,  there exists $\x_n \in O_{n,0}$ such that $f^i(\x_n) \in O_{n,i}$. Furthermore, as $\mathcal{A}$ is complete, we have that $(\x_n) \rightarrow \x$ for some $\x \in X$. We note that $\x = \beta ((O_n))$.  From the continuity of $f$ we have that $f^i(\x_n)\rightarrow f^i(\x)$. On the other hand $f^i(\x_n) \rightarrow \bigcap_n O_{n,i}$. Therefore $f^i(\x)=\bigcap_n O_{n,i}$.  Since the $i$ fixed was arbitrary we obtain that $f^i(\x)=\bigcap_n O_{n,i}$ for all $i$. Hence $\theta \circ \beta = id$ on $\invlim{\hookrightarrow}{\m{O}{}(\m{U}{n})}$.

$(ii)$ Assume that $\mathcal{A}$ is tame. If $f$ is uniformly continuous then, by Proposition~\ref{fullmaps} $(ii)$, we have that $\theta$ is uniformly continuous. The map $\beta = \theta ^{-1}$ is uniformly continuous as it is the composition of the isometry $\alpha^{-1}$ and the 1-Lipschitz map $\pi$.

    Lastly, assume that $\theta$ and $\theta^{-1}$ are uniformly continuous. By Lemma~\ref{bebavinhonaquarentena}, we have that $f = \theta^{-1} \circ \sigma^* \circ \theta$. As $\sigma^*$ is 1-Lipschitz, we conclude that $f$ is uniformly continuous.
\end{proof}

The following theorem follows from Lemma~\ref{bebavinhonaquarentena} and Proposition~\ref{contTheta}.

\begin{Theorem}\label{basicrep} Let $(X,d,f)$ be a dynamical system and $\mathcal{A} =\{\m{U}{n}\}_{n \in \N}$ be a  complete defining sequence of $(X,d)$. 

\begin{enumerate}
    \item $(X,d,f)$ is topologically conjugate to the inverse limit of $(\hookrightarrow,(\m{O}{} (\m{U}{n}),\sigma_n))$,
equipped with the shift map $ \sigma^*  : =(\sigma_n)^*$.
\item If, in addition, $\mathcal{A}$ is tame and $f$ is uniformly continuous, then the conjugacy can be made uniform.
\end{enumerate}
\end{Theorem}

The following proposition gives a useful characterization of the finite shadowing property.

\begin{Proposition}\label{shadowreform}
Let $(X,d,f)$ be a dynamical system and $\mathcal{A} =\{\m{U}{n}\}_{n \in \N}$ be a tame defining sequence of $(X,d)$. Then, $f$ has the finite shadowing property if, and only if, for every $m\in \N$, there is $n > m$ such that $\m{PO}{}(\m{U}{n}) \hookrightarrow  \m{O}{}(\m{U}{m}) $.
\end{Proposition}
\begin{proof}
Let $\mathcal{A} =\{\m{U}{n}\}_{n \in \N}$ be a tame defining sequence of $(X,d)$ and let $\rho _n>0$, $n \in \N$, be such that $\m{U}{n}$ is $\rho_n$-separated.

Suppose that $f:X\to X$ has the finite shadowing property. Given $m \in \N$ choose $\varepsilon >0$ such that $\varepsilon < \rho_m$. Let $0 < \delta<\varepsilon$ be such that any finite $\delta$-pseudo-orbit is $\varepsilon$-shadowed, and take $n>m$ such that $S_n<\delta$. 

Let $(O_{i})_{i\in\N}$ be a sequence in $\m{PO}{}(\m{U}{n})$. Then, $(O_{i}) \hookrightarrow (V_i)$ where $V_i \in \m{U}{m}$. We have to prove that $(V_i) \in \m{O}{}(\m{U}{m})$.

By the definition of $\m{PO}{}(\m{U}{n})$, there exists a sequence $(\x_i)$ in $X$ such that $\x_0\in O_0$ and $f(\x_i), \x_{i+1} \in O_{i+1}$ for all $i\in \N$ (so $\x_i\in V_i$ for $i\in \N$). Since $\text{diam}(O_i)\leq S_n<\delta$ for every $i$, we have that $(\x_i)$ is a $\delta$-pseudo-orbit. Given $k>0$, let $\z \in X$ be a point that $\varepsilon$-shadows the sequence $(\x_i)_{i=0}^{k}$, i.e., such that $d(f^i(\z), \x_i)< \varepsilon$ for all $i=0, \ldots k$. Since $\x_i\in O_i \subseteq V_i$ for all $i$, the definition of $\rho_n$ and choice of $\varepsilon$ imply that $f^i(\z)\in  V_i$, for $i=1,\ldots, k$. Hence $(V_i) \in \m{O}{}(\m{U}{m})$ as desired.

We now prove the converse. Suppose that for each $m \in \N$ there is $n > m$ such that $\m{PO}{}(\m{U}{n}) \hookrightarrow  \m{O}{}(\m{U}{m})$. Given $\varepsilon>0$, take $m\in\N$ such that $S_m<\varepsilon$. Let $n>m$ be such that $\m{PO}{}(\m{U}{n}) \hookrightarrow  \m{O}{}(\m{U}{m})$ and take $\delta< \rho_n$.

Let $(\x_i)_{i=1}^k$ be a finite $\delta$-pseudo-orbit. Then $(\y_i)_{i\in \N}$, where $\y_i=\x_i$ for $i=0,\ldots,k$ and $\y_{k+j} = f^j(\x_k)$, for $j=1, 2, \ldots$, is a $\delta$-pseudo-orbit. For each $i\in \N$, let $O_i\in \m{U}{n}$ be such that $\y_i \in O_i$. Since $d(f(\y_i),\y_{i+1})<\delta<\rho_n$ we have, from the definition of $\rho_n$, that $f(\y_i),\y_{i+1} \in O_i$ for all $i$. Hence $(O_i)\in \m{PO}{}(\m{U}{n})$. Let $(V_i)\in \m{O}{}(\m{U}{m})$
be such that $ (O_i) \hookrightarrow (V_i)$. Let $\z\in X$ be such that $f^i(\z) \in V_i$ for each $i=0,\ldots, k$. Then for all $i=0,\ldots, k$ both $f^i(\z)$ and $\x_i$ belong to $V_i$. Since $\text{diam}(V_i)<S_m<\varepsilon$ we conclude that $\z$ $\varepsilon$-shadows $(\x_i)_{i=1}^k$ and hence $f$ has the finite shadowing property.

\end{proof}

 We now prove our characterization of the finite shadowing property in terms of inverse limits.

\begin{Theorem}\label{ThmInvLimSh}
Let $(X,d,f)$ be a dynamical system and $\mathcal{A} =\{\m{U}{n}\}_{n \in \N}$ be a complete tame defining sequence of $(X,d)$. Furthermore, assume that $f$ has the finite shadowing property.
\begin{enumerate}
    \item The map $f$ is conjugate to the inverse limit of a sequence of 1-step shifts on a countable alphabet. More precisely, $(X,d,f)$ is conjugate to $   (\invlim{\hookrightarrow}{\m{PO}{}(\m{U}{n})},\sigma^*)$ 
    \item Moreover, if $f$ is uniformly continuous, then the conjugacy can be made uniform. 
\end{enumerate}
\end{Theorem}
\begin{proof}
As $\mathcal{A}$ is a tame defining sequence of $X$, by Proposition~\ref{shadowreform}, for each $m \in \N$, we can choose $n >m $ such that 
\[\m{PO}{}(\m{U}{n}) \hookrightarrow  \m{O}{}(\m{U}{m}) .\] 
Moreover, for each $k \in \N$,  $\m{O}{}(\m{U}{k}) \subseteq  \m{PO}{}(\m{U}{k}) $. Hence, we can choose increasing sequences of positive integer 
\[m_1 < n_1 < m_2 < n_2\ldots\]
such that 
\[ \tag{a} \m{O}{}(\m{U}{m_1}) \hookleftarrow  \m{PO}{}(\m{U}{n_1}) \hookleftarrow \m{O}{}(\m{U}{m_2}) \hookleftarrow  \m{PO}{}(\m{U}{n_2}) \ldots
\]
The mapping on this inverse limit space $(a)$ is the shift map on each coordinate space. Now consider the following two factors of the above dynamical system:
\[ \tag{b} \m{O}{}(\m{U}{m_1}) \hookleftarrow  \m{O}{}(\m{U}{m_2}) \hookleftarrow   \ldots
\]
\[ \tag{c} \m{PO}{}(\m{U}{n_1}) \hookleftarrow   \m{PO}{}(\m{U}{n_2}) \hookleftarrow   \ldots
\]
Each of the dynamical system $(b)$ and the dynamical system $(c)$ is uniformly conjugate to the dynamical system $(a)$, and hence $(b)$ is uniformly conjugate to $(c)$. As $\mathcal{A}$ is complete, by Theorem~\ref{basicrep}, the dynamical system $(b)$ is conjugate to $(X,d,f)$.  By Lemma~\ref{basicdeflem}, each $\m{PO}{}(\m{U}{n_i})$ is 1-step shift. Hence, we have that the dynamical system $(X,d,f)$ is conjugate to the dynamical system 
\[ \tag{c} \m{PO}{}(\m{U}{n_1}) \hookleftarrow   \m{PO}{}(\m{U}{n_2}) \hookleftarrow   \ldots,
\]
consisting of 1-step shift spaces. Now note that $(c)$ is actually uniformly conjugate to the system
\[ \m{PO}{}(\m{U}{1}) \hookleftarrow   \m{PO}{}(\m{U}{2}) \hookleftarrow   \ldots,
\]
completing the proof of $(i)$.

$(ii)$ Suppose that $f$ is uniformly continuous. Then, by Theorem~\ref{basicrep}, $(b)$ is actually uniformly conjugate to $(X,d,f)$. Now the proof proceeds as in $(i)$ since all other conjugacies in question are uniform conjugacies.
\end{proof}

We now describe the relationship between the shadowing property and the finite shadowing property below.

\begin{Theorem}\label{joelho} Let $(X,d,f)$ be a dynamical system and $\mathcal{A} =\{\m{U}{n}\}_{n \in \N}$ be a tame defining sequence of $(X,d)$.
\begin{enumerate}
  \item \label{quadril} If $f$ has the shadowing property then $f$ has the finite shadowing property and $ (\hookrightarrow, \m{PO}{}(\m{U}{n}))$ satisfies the Mittag-Leffler Condition.

    \item Suppose that $\mathcal A$ is complete and $f$ is uniformly continuous. If $f$ has the finite shadowing property, and $ (\hookrightarrow, \m{PO}{}(\m{U}{n}))$ satisfies the Mittag-Leffler Condition, then $f$ has the shadowing property.
  
\end{enumerate} 

\end{Theorem}

\begin{proof}

Suppose that $f$ has the shadowing property. 
Let $N\in \N$ and pick $\delta$ that witnesses the shadowing property for $\rho_N$. Choose $k\geq N$ such that $S_k <\delta$ and let $i\geq k$. Let $(O_j)\in \m{PO}{}(\m{U}{k})$ and $(V_j)\in \m{PO}{}(\m{U}{N})$ be such that $(O_j) \hookrightarrow (V_j)$. We have to show that 
there exists $(W_j)\in \m{PO}{}(\m{U}{i})$ such that  $(W_j) \hookrightarrow (V_j)$.

From the definition of $\m{PO}{}(\m{U}{k})$ we can find an $S_k$ pseudo-orbit $(\x_0, \x_1,...)$ such that $\x_j \in O_j$ for all $j$. By the shadowing property of $f$, there is a $\z$ that $\rho_N$ shadows $(\x_0, \x_1,...)$, that is, $d(f^j(\z),\x_j)<\rho_N$ for all $j$. Notice that the orbit of $\z$ determines an element in $\m{PO}{}(\m{U}{i})$. More precisely, let $W_i \in \{\m{U}{i}\} $ be such that $f^i(\z) \in W_i$. Finally, notice that $\x_j\in V_j$ and, since $d(f^j(\z),\x_j)<\rho_N$, we have that $f^j(\z)\in V_j$ and hence $(W_j) \hookrightarrow (V_j)$ as desired.

Now suppose that $\mathcal A$ is complete.  By Theorem~\ref{ThmInvLimSh}, we have that $f$ is uniformly conjugate to $   (\invlim{\hookrightarrow}{\m{PO}{}(\m{U}{n})},\sigma^*)$, an inverse limit of 1-step shifts satisfying the Mittag-Leffler Condition. By Proposition~\ref{shadow-finiteshadow-porder} each of these 1-step shifts has the shadowing property. By Theorem~\ref{caipirinha}, we have that the inverse limit space of these 1-step shifts has the shadowing property. As of now we have that $f$ is uniformly conjugate to a space with the shadowing property. We conclude, by Proposition~\ref{UnifCong}, that $f$ has the shadowing property.

\end{proof}

\begin{Corollary}\label{CorMain} 
Let $(X,d,f)$ be a dynamical system and $\mathcal{A} =\{\m{U}{n}\}_{n \in \N}$ be a complete tame defining sequence of $(X,d)$. Furthermore, assume that $f$ is uniformly continuous. Then, $f$ has the shadowing property if, and only if, $f$ is uniformly conjugate to an inverse limit dynamical system, in which the associated inverse system consists of a sequence of 1-step shifts on a countable alphabet, with uniformly continuous bonding maps, and satisfies the Mittag-Leffler Condition.
\end{Corollary}
\begin{proof}

If $f$ has the shadowing property then just apply Theorem~\ref{ThmInvLimSh} and Theorem~\ref{joelho}.

For the converse, notice that by Proposition~\ref{shadow-finiteshadow-porder} we have that each of the 1-step shifts in the inverse limit has the shadowing property. By Theorem~\ref{caipirinha}, we have that the inverse limit space of these 1-step shifts has the shadowing property. As of now we have that $f$ is uniformly conjugate to a space with the shadowing property. We conclude, by Proposition~\ref{UnifCong}, that $f$ has the shadowing property. 
\end{proof}

\begin{Remark}
Example~\ref{FinShadowEx} shows that Corollary~\ref{CorMain} is sharp, i.e., one cannot drop the hypothesis of $f$ being uniformly continuous. 
\end{Remark}

Next we describe a class of functions-spaces where the shadowing property and the finite shadowing property agree. More precisely, we have the following.

\begin{Proposition} \label{balance}
 Let $(X,d,f)$ be a dynamical system and $\mathcal{A} =\{\m{U}{n}\}_{n \in \N}$ be a  tame defining sequence of $(X,d)$. Furthermore, assume that for some $n$,  the partition $\m{U}{n}$ consists of compact sets. Then a uniformly continuous map $f : X \rightarrow X $ has the shadowing property if, and only if, it has the finite shadowing property. 
\end{Proposition}
\begin{proof}

Recall that $u_\mathcal A$ denotes the metric associated to the defining sequence $\mathcal A = \{\m{U}{n}\}_{n \in \N}$ as in Equation~(\ref{polvo}), and notice that the existence of a partition $\m{U}{n}$ consisting of compact sets is equivalent to $(X,u_\mathcal A) $ be a uniformly locally compact space. By Proposition~\ref{berries}, $u_\mathcal A$ is uniformly equivalent to $d$. Since the notion of uniformly locally compact space is preserved by uniform equivalence, we conclude that if $(X,d)$ is a metric space with a tame defining sequence, then $(X,d)$ is uniformly locally compact. The result of the proposition now follows from Proposition~\ref{UnifLCThm}.


\end{proof}

The following proposition guarantees that the Mittag-Leffler Condition is satisfied for certain nice defining sequences. 

\begin{Corollary}
\label{mouse}
 Let $(X,d,f)$ be a dynamical system and $\mathcal{A} =\{\m{U}{n}\}_{n \in \N}$ be a  tame defining sequence of $(X,d)$. Furthermore, assume that $f$ is uniformly continuous and satisfies the finite shadowing property.  If $\m{U}{n}$ is a partition consisting of compact sets, for some $n\in \N$, then $(\hookrightarrow,\m{PO}{}(\m{U}{n},f))$ satisfies the Mittag-Leffler Condition. 
\end{Corollary}
\begin{proof} 

By Proposition~\ref{balance} we have that $f$ satisfies the shadowing property. Hence, by item~\ref{quadril} in Theorem~\ref{joelho}, we conclude that $(\hookrightarrow,\m{PO}{}(\m{U}{n},f))$ satisfies the Mittag-Leffler Condition.


\end{proof}

\section{Applications}\label{aplica}

\subsection{Shadowing in ultrametric spaces}

In this subsection, as applications of the techniques and results developed so far, we prove that various classes of maps in ultrametric spaces have the shadowing property.

\begin{Definition}\label{feijao}

We call $f: X \rightarrow X$ an {\em eventually Lipschitz $L$ map} if there is $ \varepsilon >0$ such that for all $x,y \in X$ with $d(x,y) < \varepsilon$ we have that $d(f(x), f(y)) \le  L \cdot d(x,y)$.

We call $f: X \rightarrow X$ an {\em eventual similarity} if there is $\varepsilon >0$  and $s >0$ such that for all $x,y \in X$ with $d(x,y) < \varepsilon$ we have that $d(f(x), f(y)) = s \cdot d(x,y)$.
\end{Definition}

We start by showing the finite shadowing property for maps that are $1$-Lipschitz or such that the inverse is $1$-Lipschitz.

\begin{Theorem}\label{Doing good} Suppose that $X$ is an ultrametric space and $f:X\rightarrow X$ is continuous. Then, under any one of the following hypothesis the map $f$ has the finite shadowing property.
\begin{enumerate}
    \item $f$ is eventually $1$-Lipschitz;
    \item $f$ is invertible and $f^{-1}$ is eventually $1$-Lipschitz;
    \item $f$ is an invertible eventual similarity and $f^{-1}$ is uniformly continuous.
\end{enumerate}
\end{Theorem}
\begin{proof}
Let $\{\mathcal{U}_n\} $ be the tame defining sequence of Proposition~\ref{cabra}. By Proposition~\ref{shadowreform},  it suffices to show that for every $m\in \N$ there is $n > m$ such that $\m{PO}{}(\m{U}{n},f) \hookrightarrow  \m{O}{}(\m{U}{m},f) $. 

We first prove the result for $f$ eventually $1$-Lipschitz. Let $\varepsilon >0$ be as in the definition of eventual Lipschitz. Let $m \in \N$. As $\m{O}{}(\m{U}{m'},f) \hookrightarrow  \m{O}{}(\m{U}{m},f)$ for $m' \ge m$, we may assume that $m$ is large enough so that $1/m < \varepsilon$. 
Now let $n=m+1$. Take $(O_i)_{i\in\N} \in \m{PO}{}(\m{U}{n},f)$, and let $V_i \in \m{U}{m} $ be such that $O_i\hookrightarrow V_i$. We have to prove that $ \forall \ k \in \N, \ \exists \ \x \in X \textit{ s.t. } f^i(\x) \in V_i, \  0 \le i  \le k  $.
We actually prove more, namely, for any $\x \in O_0$, we have that $f^i(\x)\in V_i $.
We  accomplish this by showing that  $f^i(O_0) \subseteq O_{i+1}$, $i \ge 1$. This, in turn, is accomplished by showing that $f(O_i) \subseteq O_{i+1}$, $i \ge 0$.
Indeed, by definition, $(O_i)_{i\in\N} \in \m{PO}{}(\m{U}{n},f)$ implies that $f(O_i)\cap O_{i+1}\neq \emptyset$. Let $f(\x)\in O_{i+1}$ with $\x\in O_i$. Then for any  $\z\in O_i$ we have that $d(f(\x), f(\z))\leq d(\x,\z) < \frac{1}{n} $, implying that $f(O_i)\subseteq B(f(\x),\frac{1}{n})$. From Property~\ref{triciclo} we have that $B(f(\x),\frac{1}{n}) = O_{i+1}$, yielding that $f(O_i) \subseteq O_{i+1}$. 

Now we consider the case when $f^{-1}$ exists and is eventually $1$-Lipschitz.  We proceed as earlier and let $n=m+1$. Take $(O_i)_{i\in\N} \in \m{PO}{}(\m{U}{n},f)$, and let $V_i \in \m{U}{m} $ be such that $O_i\hookrightarrow V_i$. We have to prove that $ \forall \ k \in \N, \ \exists \ \x \in X \textit{ s.t. } f^i(\x) \in V_i, \  0 \le i  \le k  $. This time we observe that for all $i \ge 0 $, we have that $f^{-1}(O_{i+1}) \subseteq O_i$.  Indeed, as $(O_i)_{i\in\N} \in \m{PO}{}(\m{U}{n},f)$ we have that there is $ \x \in  O_{i+1} $ with $f^{-1}(\x) \in (O_{i})$. Then for any  $\z\in O_{i+1}$ we have that $d(f^{-1}(\x), f^{-1}(\z))\leq d(\x,\z) < \frac{1}{n} $, implying that $f^{-1}(O_{i+1})\subseteq B(f^{-1}(\x),\frac{1}{n})$. From Property~\ref{triciclo} we have that $B(f^{-1}(\x),\frac{1}{n}) = O_{i}$, and hence $f^{-1}(O_{i+1}) \subseteq O_i$ for all $i \ge 0$. Now this fact and induction implies that for all $k \ge 1$,
\[ f^{-k}(O_k) \subseteq f^{k-1}(O_{k-1}) \subseteq \ldots \subset f^{-1}(O_1) \subseteq O_0.
\]
Let $\x \in f^{-k}(O_k)$. Then, for all $0\le i \le k$, we have that $f^i(\x) \in O_i \subseteq V_i$, completing the proof. 

Let us now show (iii). Let $f$ be as in the hypothesis, with the similarity constant $s$. The case of $s \le 1$ is covered in part (i). Hence let us assume that $s>1$. Let $\varepsilon$ be as in the definition of eventual similarity. Let $\delta >0$ witness the uniform continuity of $f^{-1}$ for $\varepsilon$. We will show that $f^{-1}$ is an eventually $1$-Lipschitz contraction with eventuality constant $\delta$. By the choice of $\delta$, for all $\x, \y \in X$ with $d(\x, \y)< \delta$, we have that $d(f^{-1}(\x),f^{-1}(\y)) < \varepsilon$. Hence, as $f$ is an eventual similarity, we have that 
\[ d(f(f^{-1}(\x)),f(f^{-1}(\y))) =  s \cdot d(f^{-1}(\x),f^{-1}(\y)),\] 
or equivalently,
\[  d(f^{-1}(\x),f^{-1}(\y)) = \frac{1}{s} \cdot d(\x,\y).\] 
We have just shown that $f^{-1}$ is an eventually $1$-Lipschitz. By (ii), we have that $f$ has the finite shadowing property.

\end{proof}

We now explore when the shadowing property holds for the above classes.

\begin{Corollary}
\label{camarao}
Let $(X,d)$ be a  complete ultrametric space. If $f:X \rightarrow X$ is an eventually  $1$-Lipschitz map, then $f$ has the shadowing property. 
\end{Corollary}
\begin{proof} In light of Theorem~\ref{Doing good}, we already have that $f$ has the finite shadowing property. In order to complete the proof, by Theorem~\ref{joelho}, it suffices to show that $(\hookrightarrow,\m{PO}{}(\m{U}{n}))$
satisfies the Mittag-Leffler Condition,  where $\{\mathcal{U}_n\} $ is the tame defining sequence of Proposition~\ref{cabra}.

Let $\varepsilon >0$ witness the fact that $f$ is eventually Lipschitz. Let $j \in \N$ be such that $\frac{1}{j} < \varepsilon$.
It will suffice to show the Mittag-Leffler Condition for $N > j$.
Let $k = N+1$ and let $i\geq k$. To verify the Mittag-Leffler Condition, it suffices to show that  for $(O_l)\in \m{PO}{}(\m{U}{k})$ with $(O_l) \hookrightarrow  (V_l)$ and $V_l \in \m{U}{N}$, there is $(W_l) \in  \m{PO}{}(\m{U}{i})$ such that  $(W_l) \hookrightarrow  (V_l)$. To this end, let $(\x_l)$ be a sequence such that $\x_l \in O_l$ and $f(\x_l)\in O_{l+1}$, $l\in \N$. 

We will show that $f^l(\x_0)\in V_l$ for all $l\in \N$. By hypothesis, $\x_l \in O_l \subseteq V_l$ and $f(\x_l)\in O_{l+1} \subseteq V_{l+1}$,  for all $l\in \N$. Clearly, $\x_0 \in V_0$. Suppose that $f^l(\x_0) \in V_l$ for every $l\leq M$. Then, as $f^M(\x_0), \x_M \in V_M \in \m{U}{N}$ and $f$ is $1$-Lipschitz, we have that 
\[d(f^{M+1}(\x_0), f(\x_M) ) \leq  d(f^M(\x_0), \x_M) <\frac{1}{N},\]
verifying that $f^{M+1}(\x_0)\in V_{M+1}$ as desired.

Now, let $W_l$ be the element in $\m{PO}{}(\m{U}{i})$ such that $f^l(\x_0) \in W_l$. Clearly, $(W_l) \hookrightarrow  (V_l)$.
\end{proof}

\begin{Corollary}\label{almoco}
Suppose that $(X,d)$ is an ultrametric space with the additional property that for some $\varepsilon >0$, all balls of radius  $\varepsilon$ are compact.  Let $f:X\rightarrow X$ be an invertible uniformly continuous map.
\begin{enumerate}
    \item If $f^{-1}$ is an eventually $1$-Lipschitz map, then $f$ has the shadowing property.
   \item If $f^{-1}$ is uniformly continuous and $f$ is an eventual similarity, then $f$ has the shadowing property. 
\end{enumerate} 
\end{Corollary}
\begin{proof}
In both cases, by Theorem~\ref{Doing good}, we have that $f$ has the finite shadowing property. Since $(X,d)$ is a uniformly locally compact ultrametric space, by Proposition~\ref{UnifLCThm} (or by Corollary~\ref{balance}) we conclude that $f$ has the shadowing property.


\end{proof}

\begin{Corollary}\label{scalingmaps}
Let $X$ be a compact ultrametric space and $f:X \rightarrow X$ be an eventual similarity. Then $f$ has the shadowing property.
\end{Corollary}
The next example shows that  Theorem~\ref{Doing good} is not valid for general Lipschitz functions.

\begin{Example}\label{exqueijo}
Let $A$ be a countable alphabet and consider the metric in $A^\N$ given in Equation~(\ref{productmetric}). This is an ultrametric space.  Let $X\subseteq A^{\N}$ be a shift space which is not of finite order. Then the shift map on $X$ is a $2$-Lipschitz map which, by Proposition~\ref{shadow-finiteshadow-porder}, does not have the shadowing property.
\end{Example}
Motivated by the question (1) left at the end of the paper \cite{BCA}, we finish this subsection with a result regarding two-sided shadowing in ultrametric spaces. The definition of two-sided shadowing is the same as Definition~\ref{mignon}, with $I=\Z$ in Definition~\ref{lasagna}.

\begin{Proposition}\label{pastel}
Let $(X,d)$ be an ultrametric space and $f:X\rightarrow X$ be a surjective isometry. Then $(X,d,f)$ has the two-sided shadowing property.
\end{Proposition}
\begin{proof}
Given $\varepsilon>0$, let $0< \delta <\varepsilon$ and $(\x_n)_{n\in \Z}$ be a $\delta$-pseudo-orbit. Notice that for $n\in \N$ we have:
\[ \begin{array}{ll} d(f^n(\x_0), \x_n) & \leq  \max\{d(f^n(\x_0),f(\x_{n-1})),  d(f(\x_{n-1}),\x_n)\}\\ \\
 & = \max\{ d(f^{n-1}(\x_0),\x_{n-1}),  d(f(\x_{n-1}),\x_n)\}. \end{array}
\]
Since $d(f(\x_0),\x_1)<\delta$ the above and induction imply that $\x_0$  $\varepsilon$-shadows $(\x_n)_{n\in\N}$.
Now notice that we also have, for $n\in \N$, that
\[ \begin{array}{ll} d(f^{-n}(\x_0), \x_{-n}) & \leq  \max\{d(f^{-n}(\x_0),f^{-1}(\x_{-n+1})),  d(f^{-1}(\x_{-n+1}),\x_{-n})\}\\ \\
 & = \max\{ d(f^{-n+1}(\x_0),\x_{-n+1}),  d(f^{-1}(\x_{-n+1}),\x_{-n})\}. \end{array}
\]
Since $d(f(\x_i),\x_{i+1})<\delta$ implies that $d(\x_i,f^{-1}(\x_{i+1}))<\delta$ the above, and another induction, proves that $(\x_n)_{n\in \Z}$ is two-sided $\varepsilon$-shadowed by $\x_0$.
\end{proof}

\begin{Remark}
It follows from the above proposition that the identity map in an ultrametric space always has the two-sided shadowing property. This is in contrast with the behavior of the identity map in a nontrivial connected space, where it does not have even the finite shadowing property.
\end{Remark}

\subsection{Shadowing in $P$-adic dynamics}

In the recent paper \cite{BCA} shadowing and structural stability in $p$-adics dynamics is studied. Below we show that some of their main results concerning shadowing follow from our general results in ultrametric spaces.  

Before we state the results, recall that a map $f: \mathbb{Z}_p \to \mathbb{Z}_p$ is $(p^{-k}, p^m)$ locally scaling ( $1 \leq m \leq k$ integers) if for all $x, y \in \mathbb{Z}_p$ with $\| x- y \|_p \leq p^{-k}$, we have that $\| f(x)- f(y) \|_p = p^{m} \| x- y \|_p$, see \cite{BCA, KLPS}.

\begin{Corollary}\label{Cortakethat}
The following hold.

\begin{enumerate}
    \item \cite[Theorem~1]{BCA} \label{porco}
If $f: \mathbb{Z}_p \to \mathbb{Z}_p$ is a $(p^{-k},p^{m})$ locally scaling function, where $1 \leq m \leq k $ are integers, then $f$ has the shadowing property.
 \item  \cite[Proposition~18]{BCA} \label{costela}
If $f: \Z_p \to \Z_p$ is a $1$-Lipschitz map, then $f$ has the shadowing property.
\item  \cite[Remark~19]{BCA} \label{farinha}
If  $f: \mathbb{Q}_p \to \mathbb{Q}_p$ is a  $1$-Lipschitz map, then $f$ has the shadowing property.
\end{enumerate}
\end{Corollary}

\begin{proof}
Since $\Z_p$ is compact, item \ref{porco} follows from Corollary~\ref{scalingmaps}. Items \ref{costela} and \ref{farinha} follow from Corollary~\ref{camarao} and the fact that $\mathbb{Q}_p$ is an ultrametric space with compact balls.

\end{proof}

Finally, at the end of \cite{BCA} the following question is left open:

{\bf Question:}  Let $f : \mathbb{Q}_p \to \mathbb{Q}_p$ be a homeomorphism. Assuming that  $f$ is $1$-Lipschitz, can $f$ be (two-sided) shadowing or Lipschitz structurally stable?

The (affirmative) answer to the shadowing part of this question follows directly from our Proposition~\ref{pastel}.\\

{\em {\bf{Acknowledgement}:} We would like to kindly thank Jonathan Meddaugh for pointing out an oversight in the earlier version of the article.

Many of the proofs in the current version were simplified and clarified by the deep insights of the anonymous referee. We are indebted to him/her  for significantly improving the presentation of our article.}

\bibliographystyle{abbrv}
\bibliography{shadowing}

\noindent Udayan B. Darji\\ ubdarj01@louisville.edu\\
{\em Department of Mathematics,\\ University of Louisville,\\
Louisville,
KY 40292, USA.}\\

\smallskip
\noindent Daniel Gon\c{c}alves\\
daemig@gmail.com\\
Departamento de Matem\'{a}tica\\
Universidade Federal de Santa Catarina \\
Florian\'{o}polis, SC 88040-900 , Brazil\\

\smallskip
\noindent Marcelo Sobottka\\
marcelo.sobottka@ufsc.br\\
Departamento de Matem\'{a}tica\\
Universidade Federal de Santa Catarina \\
Florian\'{o}polis, SC 88040-900 , Brazil\\
\end{document}